\pdfoutput=1
\documentclass[a4paper,english]{jnsao} 
\usepackage[utf8]{inputenc}
\usepackage[T1]{fontenc}
\usepackage{csquotes}
\usepackage{babel}
\usepackage{tikz}
\usepackage{float}
\usepackage{algpseudocode}
\usepackage{enumitem}
\usepackage{comment}
\usepackage{mathtools}
\usepackage{subcaption}
\usepackage[nameinlink,capitalize]{cleveref}
\usepackage[textsize=footnotesize,color=DarkOrange!40]{todonotes}
\usepackage{ifdraft}
\newcommand{\term}{\emph}

\newcommand{\field}[1]{\mathbb{#1}}
\newcommand{\N}{\mathbb{N}}

\newcommand{\R}{\field{R}}
\newcommand{\extR}{\overline \R}
\newcommand{\B}{B}

\newcommand{\norm}[1]{\|#1\|}
\newcommand{\adaptnorm}[1]{\left\|#1\right\|}

\newcommand{\abs}[1]{|#1|}

\newcommand{\inv}[1]{#1^{-1}}
\newcommand{\grad}{\nabla}
\newcommand{\freevar}{\,\boldsymbol\cdot\,}

\newcommand{\Union}\bigcup
\newcommand{\Isect}\bigcap
\newcommand{\union}\cup
\newcommand{\isect}\cap
\newcommand{\bigunion}\bigcup
\newcommand{\bigisect}\bigcap

\newcommand{\defeq}{:=}

\newcommand{\downto}{\searrow}
\newcommand{\upto}{\nearrow}
\newcommand{\subdiff}{\partial}

\newcommand{\MIN}[1]{{\underline {#1}}}
\newcommand{\MAX}[1]{{\overline {#1}}}

\DeclareMathOperator*{\argmin}{arg\,min}

\DeclareMathOperator{\closure}{cl}

\DeclareMathOperator{\Dom}{dom}

\def \uminusSym{\setbox0=\hbox{$\cup$}\rlap{\hbox 
        to\wd0{\hss\raise0.5ex\hbox{$\scriptscriptstyle{-}$}\hss}}\box0}
    
\newcommand{\iprod}[2]{\langle #1,#2\rangle}

\def\llangle{\langle\kern-3pt\langle}
\def\rrangle{\rangle\kern-3pt\rangle}
\newcommand{\pointwiseiprod}[2]{\llangle #1,#2\rrangle}

\def \weaktostarSym{\setbox0=\hbox{$\rightharpoonup$}\rlap{\hbox 
        to\wd0{\hss\raise1ex\hbox{$\scriptscriptstyle{*\,}$}\hss}}\box0}
    
\def\linear{\mathbb{L}}

\newcommand{\setto}{\rightrightarrows}

\def\extR{\overline \R}

\def\opt#1{\bar #1}
\def\realopt#1{\hat #1}
\def\this#1{#1^k}
\def\nexxt#1{#1^{k+1}}

\def\optu{{\opt{u}}}
\def\optx{{\opt{x}}}

\def\opty{{\opt{y}}}

\def\realoptu{{\realopt{u}}}

\def\realoptx{{\realopt{x}}}

\def\realopty{{\realopt{y}}}

\def\nextu{\nexxt{u}}
\def\nextx{\nexxt{x}}
\def\nextz{\nexxt{z}}
\def\nexty{\nexxt{y}}

\def\thisu{\this{u}}
\def\thisx{\this{x}}
\def\thisz{\this{z}}
\def\thisy{\this{y}}
\def\thisv{\this{v}}

\def\tauTest{\phi}

\def\sigmaTest{\psi}

\def\GenGap{\mathcal{G}}
\def\GenLag{\mathcal{L}}

\newcommand{\Test}{Z}
\newcommand{\Precond}{M}

\def\MAX{\overline}
\def\MIN{\underline}

\def\thisz{\this{z}}
\def\nextz{\nexxt{z}}

\DeclareMathOperator{\prox}{prox}

\def\d{\,d}

\newcommand{\Meas}{\mathcal{M}}

\newcommand{\fakenorm}[1]{\llbracket #1 \rrbracket}

\def\infconv{\mathop{\Box}}

\let\phi=\varphi
\let\epsilon=\varepsilon
\def\alt#1{\tilde #1}

\def\PpredictConstr{\mathcal{B}}
\def\PDpredictConstr{\mathcal{U}}
\def\InvDisplacements{\mathcal{V}}
\def\Id{\mathop{\mathrm{Id}}}

\def\BVspace{\mathop{\mathrm{BV}}}

\let\gapmod=\breve

\DeclareMathOperator{\regret}{regret}
\DeclareMathOperator{\dynregret}{dynamic\_regret}

\def\PredictPenalty{p}

\renewrobustcmd{\downto}{{{\mathchoice%
            {\rotatebox[origin=c]{-20}{$\to$}}
            {\rotatebox[origin=c]{-20}{$\to$}}
            {\rotatebox[origin=c]{-20}{\scalebox{0.75}{$\to$}}}
            {\rotatebox[origin=c]{-20}{\scalebox{0.6}{$\to$}}}
}}}

\renewrobustcmd{\upto}{{{\mathchoice%
            {\rotatebox[origin=c]{20}{$\to$}}
            {\rotatebox[origin=c]{20}{$\to$}}
            {\rotatebox[origin=c]{20}{\scalebox{0.75}{$\to$}}}
            {\rotatebox[origin=c]{20}{\scalebox{0.6}{$\to$}}}
}}}

\usepackage{tikz,pgfplots}
\usepackage{filmstrip}

\ifdefined\arxivprebuild
    \usetikzlibrary{external}
    \tikzexternaldisable
\else
    \usepackage{tikzexternal}
    \tikzsetexternalprefix{build/}
    \let\tikzexternaldisable\null
    \let\tikzexternalenable\null
\fi

\usepgfplotslibrary{colorbrewer}
\pgfplotsset{
    table/search path={img/},
    ignore legend/.style={every axis legend/.code={\renewcommand\addlegendentry[2][]{}}}
}%

\input{predictmacros}

\newtheorem{assumption}[definition]{Assumption}

\floatstyle{ruled}
\floatname{algorithm}{Algorithm}
\newfloat{algorithm}{tbp!}{loa}
\crefname{algorithm}{Algorithm}{Algorithms}
\crefname{assumption}{Assumption}{Assumptions}
\numberwithin{algorithm}{section}

\manuscriptcopyright{}
\manuscriptlicense{}

\title{Predictive online optimisation with applications to optical flow}
\shorttitle{Predictive online optimisation}
\author{
    Tuomo Valkonen\thanks{Department of Mathematics and Statistics, University of Helsinki, Finland \emph{and} ModeMat, Escuela Politécnica Nacional, Quito, Ecuador, \email{tuomo.valkonen@iki.fi}, \orcid{0000-0001-6683-3572}}
    }
\shortauthor{T.~Valkonen}

\date{2020-02-07 (revised 2020-08-03)}

\acknowledgements{
    This research has been supported by Escuela Politécnica Nacional internal grant PIJ-18-03 and Academy of Finland grants 314701 and 320022.
}

\begin{document}

\maketitle

\begin{abstract}
    Online optimisation revolves around new data being introduced into a problem while it is still being solved; think of deep learning as more training samples become available. We adapt the idea to dynamic inverse problems such as video processing with optical flow. We introduce a corresponding \emph{predictive online primal-dual proximal splitting} method. The video frames now \emph{exactly correspond to the algorithm iterations}. A user-prescribed predictor describes the evolution of the primal variable. To prove convergence we need a predictor for the dual variable based on (proximal) gradient flow. This affects the model that the method asymptotically minimises. We show that for inverse problems the effect is, essentially, to construct a new dynamic regulariser based on infimal convolution of the static regularisers with the temporal coupling.  We finish by demonstrating excellent real-time performance of our method in computational image stabilisation and convergence in terms of regularisation theory.
\end{abstract}

\section{Introduction}
\label{sec:intro}

On Hilbert spaces $X_k$ and $Y_k$, ($k \in \N$), consider the formal problem
\begin{equation}
    \label{eq:intro:problem-sequence}
    \min_{x^1,x^2,\ldots}~ \sum_{k=1}^\infty F_k(x^k)+G_k(K_k x^k)
    \quad\text{s.t.}\quad
    \nextx=\opt A_k(\thisx),
\end{equation}
where $F_k: X_k \to \extR$ and $G_k: Y_k \to \extR$ are convex, proper, and lower semicontinuous, $K_k \in \linear(X_k; Y_k)$ is linear and bounded, and the \term{temporal coupling} operators $\opt A_k: X_k \to X_{k+1}$.
One may think of $\min(F_k+G_k \circ K_k)$ as a problem we want to solve on each time instant $k$, knowing that the solutions of these problems are coupled via the environment acting through $\opt A_k$. For example, $\opt A_k$ can describe the true movement of objects in a scene, that we cannot control, and do not necessarily know.
This problem is clearly challenging; even its solutions are generally well-defined only asymptotically.

Instead of trying to solve \eqref{eq:intro:problem-sequence} exactly, what if we take only \emph{one step} of an optimisation algorithm on each partial problem
\begin{equation}
    \label{eq:intro:problem-static}
    \min_{x^k \in X_k}~ J_k(x_k) \defeq F_k(x^k)+G_k(K_k x^k),
\end{equation}
and use an approximation $A_k: X_k \to X_{k+1}$, called the \term{predictor}, of the  unknown $\opt A_k$ to transfer iterates between the steps? Can we obtain convergence in an asymptotic sense, and to what? We set out to study these questions, in particular to develop a predictive “online” primal-dual method.

Our simple model problem is image sequence denoising: we are given noisy images
$\{b^k\}_{k \in \N}$ in the space\footnote{The total variation term in \eqref{eq:intro:rof} in principle requires $x \in \BVspace(\Omega)$, the space of functions of bounded variation on $\Omega$. This is not a Hilbert space, but merely a Banach space, where our overall setup \eqref{eq:intro:problem-sequence} does not to apply. However, due to the weak(-$*$) lower semicontinuity of convex functionals, any minimiser of \eqref{eq:intro:rof} necessarily lies in $L^2(\Omega) \isect \BVspace(\Omega)$, so we are justified in working in the Hilbert space $X=L^2(\Omega)$, and seeing $\BVspace(\Omega)$ as a constraint imposed by the total variation term.} $X = L^2(\Omega)$ on the two-dimensional domain $\Omega \subset \R^2$, and bijective displacement fields $v^k: \Omega \to \Omega$ such that the images roughly satisfy the \term{optical flow} constraint $b^{k+1} \approx A_k(b^k)$ for $A_k(x) \defeq x \circ v^k$.
For an introduction to optical flow, we refer to \cite{becker2015opticalflow}.
The static problem \eqref{eq:intro:problem-static} is the isotropic total variation denoising
\begin{equation}
    \label{eq:intro:rof}
    \min_{x \in X}~ \frac{1}{2}\norm{x-b^k}_X^2 + \alpha\norm{Dx}_{\Meas},
\end{equation}
where $\alpha>0$ is a regularisation parameter and $D$ a measure-valued differential operator.
In the dynamic case we would like the approximate solutions $\{\thisx\}_{k \in \N}$ to also satisfy $\nextx \approx A_k(\thisx)$.
In principle, we could for the first $N$ frames for some penalisation parameter $\beta>0$ solve
\[
    \min_{x^1,\ldots,x^{N+1} \in X}~ \sum_{k=0}^N\left(
        \frac{1}{2}\norm{\thisx-b^k}_X^2 + \alpha\norm{Dx^k}_{\Meas}
        + \frac{\beta}{2}\norm{\nextx-A_k(\thisx)}_X^2
    \right),
\]
or a version that linearises $A_k$.
However, when the number of frames $N$ is high, these problems become numerically increasingly challenging. Also, if we want to solve the problem for $N+1$ frames, we may need to do the same amount of work again, depending on how well our algorithm can “restart”. Primal-dual methods in particular tend to be very sensitive to initialisation.

An alternative is to try to solve the problem in an “online” fashion, building the gradually changing data into the algorithm design \cite{zinkevich2003online}. We refer to \cite{hazan2016introduction,belmega2018online,orabona2020modern} for introductions and further references to online methods in machine learning. Online Newton methods have also been studied for smooth PDE-constrained optimisation \cite{biegler2007real,groetschel2013online}.
Our approach has more in common with machine learning and nonsmooth optimisation.
From this point of view, basic online methods seek a low \term{regret} for a dynamic solution sequence compared to a fixed solution. With the notation $x^{1:N} \defeq (x^1,\ldots,x^N)$, for any comparison set $B \subset X$, where we expect the true solution to lie, we define the regret as
\[
    \regret_B(x^{1:N}) \defeq \sup_{\optx \in B}~ \sum_{k=1}^N\left( J_k(\thisx) - J_k(\optx)\right).
\]
This  does not model the temporal nature of our problem, so in \cite{hall13dynamical} \term{dynamic regret} is introduced. For a \term{comparison set} $\PpredictConstr_{1:N} \subset \prod_{k=1}^N X_k$ of potential true solutions, it reads
\begin{equation}
    \label{eq:intro:dynregret}
    \dynregret_{\PpredictConstr_{1:N}}(x^{1:N}) \defeq \sup_{\optx^{1:N} \in \PpredictConstr_{1:N}}~ \sum_{k=1}^N\left( J_k(\thisx) - J_k(\this\optx)\right).
\end{equation}
For example, we can take
\begin{equation}
    \label{eq:intro:PpredictConstr}
    \PpredictConstr_{1:N} = \{ (\optx^1,\ldots, \optx^N) \mid
        \optx^0 \in \PpredictConstr_0,\, \nexxt\optx=\opt A_k(\this\optx),\, k = 0,\ldots,N-1\}
\end{equation}
for some $\PpredictConstr_0 \subset X_0$, where we expect the initial true $\optx^0$ to lie, and the true temporal coupling operators $\opt A_k: X_k \to X_{k+1}$.
For the optical flow problem, \eqref{eq:intro:PpredictConstr} would read
\[
    \PpredictConstr_{1:N} = \{ (\optx^0 \circ \opt v_1, \ldots, \optx^0 \circ \opt v_1 \circ \cdots \circ \opt v_N) \mid
        \optx^0 \in \PpredictConstr_0\}
\]
for some true displacement fields $\opt v_k$ and a set $\PpredictConstr_0$ containing the initial non-corrupted frame $\optx^0$. Thus $\PpredictConstr_{1:N}$ consists of all potential “true” frames $\opt x^1,\ldots,\opt x^N$ generated from all potential initial $\optx^0$ by the true displacement fields.
When the dynamic regret \eqref{eq:intro:dynregret} is below zero, the algorithmic iterates $x^{1:N}$ fit the data and total variation regularisation of \eqref{eq:intro:rof} better than all $\optx^{1:N} \in \PpredictConstr_{1:N}$, but may not satisfy the constraint $x^{1:N}=(x^0 \circ v_1, \ldots, x^0 \circ v_1 \circ \cdots \circ v_N)$ for \emph{any} displacement fields $v_k$. Specific algorithms may additionally seek to approximately satisfy this constraint for some measured or estimated displacement fields $v_k$.

The idea now would be to obtain a low dynamic regret by some strategy. One possibility is what we already mentioned: take one step of an optimisation method towards a minimiser of each $J_k$, and then use $A_k$ to predict an approximate solution for the next problem. Repeat. In this approach, data frames exactly correspond to algorithm iterations.
The strategy of very inexact solutions is motivated by the fact that neural networks can be effective---not get stuck in local optima---because subproblems are not solved exactly \cite{bousquet2008tradeoffs}. A different type of applications with only intermittent sampling is studied in \cite{bastianillo2020prediction,simonetto2017time}

In \cref{sec:fb} we we prove low dynamic regret for predictive forward-backward splitting, in line with  the literature \cite{hall13dynamical,zhang2019distributed}. This serves to introduce concepts and ideas for our main interest: primal-dual methods.
Indeed, forward-backward splitting is poorly applicable to \eqref{eq:intro:rof}: the proximal step is just as expensive as the original problem. It is more effective on the dual problem, however, we are given a primal predictor $A_k$. Moreover, purely dual formulations are not feasible for deblurring and more complex inverse problems.
A solution is to work with primal-dual formulations of the static problems \eqref{eq:intro:problem-static},
\begin{equation}
    \label{eq:intro:minmax-sequence}
    \min_{x \in X_k}\max_{y \in Y_k}~F_k(x)+\iprod{K_kx}{y}-G_k^*(y).
\end{equation}
Here $G_k^*$ is the Fenchel conjugate of $G_k$. A popular method for this type of problems is the \term{primal-dual proximal splitting} (PDPS) of Chambolle and Pock \cite{chambolle2010first}. We refer to \cite{tuomov-firstorder} for an overview of variants, alternatives, and extensions to non-convex problems.

\subsection*{Main contributions}

We develop in \cref{sec:pd} a predictive online PDPS for \eqref{eq:intro:problem-sequence}.
For the primal variable we use the user-prescribed predictor $A_k: X_k \to X_{k+1}$, but for the dual variable the regret theory imposes a more technical predictor. This forms the main challenge of our work. To prepare for this, we introduce in \cref{sec:gap} appropriate \term{partial primal gap} functionals  to replace the dynamic regret \eqref{eq:intro:dynregret}, not applicable to primal-dual methods.

We finish in \cref{sec:flow} with computational image stabilisation based on optical flow and online optimisation. We obtain real-time performance and show convergence of the algorithmic solutions in terms of regularisation theory \cite{engl2000regularization} as the noise level decreases.
Before this we introduce notation.

\subsection*{Notation}

We write $x^{n:m} \defeq (x^n,\ldots,x^m)$ with $n \le m$, and $x^{n:\infty} \defeq (x^n,x^{n+1},\ldots)$. We slice a set $\PpredictConstr \subset \prod_{k=0}^\infty X_k$ as $\PpredictConstr_{n:m} \defeq \{x^{n:m} \mid x^{0:\infty} \in \PpredictConstr\}$ and $\PpredictConstr_n \defeq \PpredictConstr_{n:n}$.
We write $\linear(X; Y)$ for the set of bounded linear operators between (Hilbert) spaces $X$ and $Y$, and $\Id \in \linear(X; X)$ for the identity operator.
We write $\iprod{x}{y}_M \defeq \iprod{Mx}{y}$ for $M \in \linear(X; X)$ and, if $M$ is positive semi-definite, also $\norm{x}_M \defeq \sqrt{\iprod{x}{x}_M}$.

We write $M \ge 0$ if $M$ is positive semidefinite and $M \simeq N$ if $\iprod{Mx}{x} = \iprod{Nx}{x}$ for all $x$.

For any $A \subset X$ and $x \in X$ we set
$
    \iprod{A}{x} \defeq \{\iprod{z}{x} \mid z \in A\}.
$
We write $\delta_A$ for the $\{0,\infty\}$-valued indicator function of $A$.
For any $B \subset \R$ (in particular $B=\iprod{A}{x}$), we use the notation $B \ge 0$ to mean that $t \ge 0$ for all $t \in B$.

For $F: X \to (-\infty, \infty]$, we write $\Dom F \defeq \{ x \in X \mid F(x) < \infty\}$ for the effective domain.
With $\extR \defeq [\infty, \infty]$ the set of extended reals, we call $F: X \to \extR$ \term{proper} if $F>-\infty$ and $\Dom F \ne \emptyset$.
Let then $F$ be convex. We write $\subdiff F(x)$ for the subdifferential at $x$ and (for additionally proper and lower semicontinuous $F$)
\[
    \prox_{F}(x) \defeq \argmin_{\alt x \in X}~ F(\alt x) + \frac{1}{2}\norm{\alt x-x}^2
    = \inv{(\Id + \subdiff F)}(x)
\]
for the proximal map. We recall that $F$ is strongly subdifferentiable at $x$ with the factor $\gamma>0$ if
\[
    F(\alt x)-F(x) \ge \iprod{z}{\alt x-x} + \frac{\gamma}{2}\norm{\alt x-x}^2
    \quad \text{for all}\quad z \in \subdiff F(x) \text{ and } \alt x \in X.
\]
In Hilbert spaces this is equivalent to strong convexity with the same factor.

Finally, for $f \in L^q(\Omega; \R^n)$, we write
$
    \norm{f}_{p,q} \defeq \adaptnorm{\xi \mapsto \norm{f(\xi)}_p}_{L^q(\Omega)}.
$

\section{Predictive online forward-backward splitting}
\label{sec:fb}

We review predictive online forward-backward splitting (POFB) for \eqref{eq:intro:problem-sequence} with $K_k=\Id$. This is useful to explain online methods in general~and to motivate our proofs and the dual comparison sequence for the online PDPS. We recall that given a step length parameter $\tau >0$, forward-backward splitting for $\min[F + G]$ iterates
\[
    \nextx \defeq \prox_{\tau G}(\thisx - \tau \grad F(\thisx)).
\]


We present a predictive online version in \cref{alg:fb:alg}. To study it, we work with:

\begin{assumption}
    \label{ass:fb:main}
    For all $k \ge 1$: $F_k, G_k: X_k \to \R$ are convex, proper, and lower semicontinuous on a Hilbert space $X_k$.
    $\grad F_k$ exists and is $L_k$-Lipschitz. We write $J_k \defeq F_k + G_k$ and $\gamma_{F_k},\gamma_{G_k} \ge 0$ for the factors of (strong) subdifferentiability of $F_k$ and $G_k$.
    We suppose for some step length parameters $\tau_k>0$ and some $\zeta_k \in (0, 1]$ that
    \begin{equation}
        \label{eq:fb:gamma-fsc}
        0 \le \gamma_k \defeq
        \begin{cases}
                \gamma_{G_k}+\gamma_{F_k}-\tau_k\inv \zeta_k L_k^2, & \gamma_{F_k}>0, \\
                \gamma_{G_k}, & \gamma_{F_k}=0 \text{ in which case we require } \tau_{k} L_{k} \le \zeta_k.
        \end{cases}
    \end{equation}

    We are also given predictors $A_k: X_k \to X_{k+1}$ and a bounded comparison set $\PpredictConstr \subset \prod_{k=0}^\infty X_k$ of potential true solutions.
    They satisfy for some (Lipschitz-like) factor $\Lambda_k$ and \term{prediction error} $\epsilon_{k+1}$ the \term{prediction bound}
    \begin{equation}
        \label{eq:fb:prediction-bound}
        \frac{1}{2}\norm{A_k(\thisx)-\nexxt\optx}^2 \le \frac{\Lambda_k}{2}\norm{\thisx-\this\optx}^2 + \epsilon_{k+1}
        \quad
        (\optx^{0:\infty} \in \PpredictConstr,\, k \in \N).
    \end{equation}
\end{assumption}

\begin{remark}
    Typically $\PpredictConstr$ is given as in \eqref{eq:intro:PpredictConstr} by some true (unknown) temporal coupling operators $\opt A_k: X_k \to X_{k+1}$ that the (known) predictors $A_k$ approximate. Then \eqref{eq:fb:prediction-bound} reads
    \[
        \frac{1}{2}\norm{A_k(\thisx)-\opt A_k(\this\optx)}^2 \le \frac{\Lambda_k}{2}\norm{\thisx-\this\optx}^2 + \epsilon_{k+1}.
    \]
    If we knew that $\opt A_k=A_k$, and the operator were Lipschitz, we could take $\Lambda_k$ as the Lipschitz factor and the prediction error $\epsilon_{k+1}=0$.
    Typically, however, we would not know the true temporal coupling---or would know it only up to measurement noise---so need the prediction errors to model this lack of knowledge or noise.
\end{remark}

\begin{algorithm}
    \caption{Predictive online forward-backward splitting (POFB)}
    \label{alg:fb:alg}
    \begin{algorithmic}[1]
        \Require For all $k \in \N$, on Hilbert spaces $X_k$, a primal predictor $A_k: X_k \to X_{k+1}$ and convex, proper, lower semicontinuous $F_{k+1}, G_{k+1}: X_{k+1} \to \extR$  such that $G_{k+1}$ has Lipschitz gradient.
        Step length parameters $\tau_{k+1}>0$.
        \State Pick an initial iterate $x^0 \in X_0$.
        \For{$k \in \N$}
            \State $\nexxt{z} \defeq A_k(\thisx)$
                \Comment{prediction}
            \State $\nextx \defeq \prox_{\tau_{k+1} G_{k+1}}(\nexxt{z} - \tau_{k+1} \grad F_{k+1}(\nexxt{z}))$
                \Comment{forward-backward step}
        \EndFor
    \end{algorithmic}
\end{algorithm}


We need to develop regret theory for \cref{alg:fb:alg}.
We recall the following \term{smoothness three-point inequalities} found in, e.g., \cite[Appendix B]{tuomov-proxtest} and \cite[Chapter 7]{clasonvalkonen2020nonsmooth}.

\begin{lemma}
    \label{lemma:smoothness}
    Suppose $F: X \to \extR$ is convex, proper, and lower semicontinuous, and has $L$-Lipschitz gradient. Then
    \begin{gather}
        \label{eq:three-point-smoothness}
        \iprod{\grad F(z)}{x-\optx}
        \ge
        F(x)-F(\optx) -  \frac{L}{2}\norm{x-z}^2
        \quad (\optx, z, x \in X).
    \end{gather}
    If $F$ is, moreover, $\gamma$-strongly convex, then for any $\beta>0$,
    \begin{gather}
        \label{eq:three-point-smoothness-sc}
        \iprod{\grad F(z)}{x-\optx}
        \ge
        F(x)-F(\optx) + \frac{\gamma-\beta L^2}{2}\norm{x-\optx}^2
            -\frac{1}{2\beta}\norm{x-z}^2
        \quad (\optx, z, x \in X).
    \end{gather}
\end{lemma}

\begin{lemma}
    \label{lemma:fb:condition}
    Suppose \cref{ass:fb:main} holds. Then, for any $k \in \N$,
    \[
        \iprod{\subdiff G_{k}(\thisx)+\grad F_{k}(\thisz)}{\thisx-\this\optx}
        \ge
        J_k(\thisx)-J_k(\this\optx)
        +\frac{\gamma_{k}}{2}\norm{\thisx-\this\optx}^2
        - \frac{\zeta_{k}}{2\tau_k}\norm{\thisx-\thisz}^2.
    \]
\end{lemma}

\begin{proof}
    If $\gamma_{F_{k}}=0$, \eqref{eq:three-point-smoothness} in \cref{lemma:smoothness} with the (strong) subdifferentiability of $G_{k}$ yield
    \[
        \iprod{\subdiff G_{k}(\thisx)+\grad F_{k}(\thisz)}{\thisx-\this\optx}
        \ge J_{k}(\thisx)-J_{k}(\this\optx)
        +\frac{\gamma_{G_{k}}}{2}\norm{\thisx-\this\optx}^2
        -\frac{L_{k}}{2}\norm{\thisx-\thisz}^2.
    \]
    Due to \eqref{eq:fb:gamma-fsc} and \cref{ass:fb:main} ensuring $\tau_{k}L_{k} \le \zeta_{k}$, this gives the claim in the case $\gamma_{F_k}=0$.
    
    If $\gamma_{F_{k}} > 0$, by \eqref{eq:three-point-smoothness-sc} for $\beta=\inv\zeta_{k}\tau_{k}$ and the (strong) subdifferentiability of $G_{k}$,
    \[
        \begin{aligned}
        \iprod{\subdiff G_{k}(\thisx)+\grad F_{k}(\thisz)}{\thisx-\this\optx}
        &
        \ge J_{k}(\thisx)-J_{k}(\this\optx)
        \\
        \MoveEqLeft[1]
        +\frac{\gamma_{G_{k}}+\gamma_{F_{k}}-\inv\zeta_{k}\tau_{k}L_{k}^2}{2}\norm{\thisx-\this\optx}^2
        -\frac{\zeta_{k}}{2\tau_k}\norm{\thisx-\thisz}^2.
        \end{aligned}
    \]
    This gives the claim by the case $\gamma_{F_k}>0$ of \eqref{eq:fb:gamma-fsc}.
\end{proof}

We now have the tools to study regret.
The sets $\PpredictConstr_{1:N}$ in the following results would typically be given by \eqref{eq:intro:PpredictConstr} through some true temporal coupling operators $\opt A_k: X_k \to X_{k+1}$.
The “testing parameters” $\tauTest_k$ can be used to derive regret rates from the regularity of the problem. We explain them in the corollary and remark to follow.

\begin{theorem}
    \label{thm:fb:main}
    Suppose \cref{ass:fb:main} holds and some \emph{testing parameters} $\{\tauTest_k\}_{k \in \N} \subset \R$ satisfy $\tauTest_{k+1} \le \tauTest_k(1+\gamma_k\tau_k)\inv\Lambda_k$ for all $k=0,\ldots,N-1$.
    Let $x^{1:N}$ generated by \cref{alg:fb:alg} for an $x^0 \in X_0$.
    Then
    \begin{multline}
        \label{eq:fb:estimate}
        \sup_{\optx^{1:N} \in \PpredictConstr_{1:N}} \sum_{k=1}^N \tauTest_k\tau_k[ J_k(\thisx) -  J_k(\this\optx)]
            +
            \sum_{k=0}^{N-1} \frac{\tauTest_{k+1}(1-\zeta_{k+1})}{2}\norm{\nextx-A_k(\thisx)}^2
        \\
        \le
        \sup_{\optx^0 \in \PpredictConstr_0} \frac{\tauTest_0(1+\gamma_0\tau_0)}{2}\norm{x^0-\optx^0}^2
        + \sum_{k=1}^N \epsilon_k\tauTest_k.
    \end{multline}
\end{theorem}

\begin{proof}
    We can write \cref{alg:fb:alg} implicitly as
    \begin{equation}
        \label{eq:fb:implicit}
        0 \in \tau_k[\subdiff G_k(\thisx)+\grad F_k(\this{z})] + (\thisx-\thisz)
        \quad (k=1,\ldots,N)
    \end{equation}
    where $\nextz \defeq A_k(\thisx)$ for $k=0,\ldots,N-1$.
    Following the testing methodology of \cite{tuomov-proxtest,clasonvalkonen2020nonsmooth}, we take any $\this\optx \in X_k$ and apply the linear “testing operator” $\tauTest_k\iprod{\freevar}{\thisx-\this\optx}$ to both sides of \eqref{eq:fb:implicit}. Following with \cref{lemma:fb:condition}, this yields
    \begin{equation}
        \label{eq:fb:tested}
        0 \ge
        \tauTest_k \iprod{\thisx-\this z}{\thisx-\this{\optx}}
        +\frac{\tauTest_k\gamma_k\tau_k}{2}\norm{\thisx-\this\optx}^2
        -\frac{\tauTest_{k}}{2}\norm{\thisx-\thisz}^2
        +\GenGap_{k}
        \quad (k=1,\ldots,N)
    \end{equation}
    for
    \[
        \GenGap_{k} \defeq \tauTest_{k}\tau_{k}[J_k(\thisx)-J_k(\this\optx)]
        + \frac{\tauTest_{k}(1-\zeta_{k})}{2}\norm{\thisx-\thisz}^2.
    \]
    We recall the Pythagoras' identity or three-point formula
    \begin{equation}
        \label{eq:fb:pythagoras}
        \iprod{\thisx-\thisz}{\thisx-\this\optx}
        = \frac{1}{2}\norm{\thisx-\thisz}^2
           - \frac{1}{2}\norm{\thisz-\this\optx}^2
           + \frac{1}{2}\norm{\thisx-\this\optx}^2.
    \end{equation}
    Hence \eqref{eq:fb:tested} yields
    \[
        \frac{\tauTest_k}{2}\norm{\thisz-\this\optx}^2
        \ge
        \frac{\tauTest_k(1+\tau_k\gamma_k)}{2}\norm{\thisx-\this\optx}^2
        +\GenGap_{k}
        \quad (k=1,\ldots,N).
    \]
    Now taking $\optx^{1:N} \in \PpredictConstr_{1:N}$ and using the prediction bound \cref{eq:fb:prediction-bound} followed by $\tauTest_{k+1}\Lambda_k \le \tauTest_k(1+\gamma_k\tau_k)$, we obtain
    \[
        \frac{\tauTest_k(1+\gamma_k\tau_k)}{2}\norm{\thisx-\this\optx}^2 + \tauTest_{k+1}\epsilon_{k+1}
        \ge
        \frac{\tauTest_{k+1}(1+\gamma_{k+1}\tau_{k+1})}{2}\norm{\nextx-\nexxt\optx}^2
        +\GenGap_{k+1}
        \quad (k=0,\ldots,N-1).
    \]
    Now we just sum over $k=0,\ldots,N-1$ and take the supremum over $\optx^{1:N} \in \PpredictConstr_{1:N}$.
\end{proof}

The next corollary, obtained with $\tauTest_k \equiv 1$ and constant $\tau_k \equiv \tau$, is similar to \cite[Theorem 4]{hall13dynamical} in the case $1+\gamma_k\tau \ge \Lambda_k$, i.e., when any available strong convexity balances the non-expansivity-like $\Lambda_k>1$ in the prediction bound \eqref{eq:fb:prediction-bound}. Often in the online optimisation literature, $\regret_B(x^1,\ldots,x^N) \le C \sqrt{N}$. The growing regret bound can arise from violating this step length condition or from the penalties $\sum_{k=1}^N \epsilon_k$ in the prediction bound \eqref{eq:fb:prediction-bound}. For our purposes, bounding the regret in terms of the initialisation and the prediction bounds is enough.

\begin{corollary}
    \label{cor:fb:dynregret}
    Suppose \cref{ass:fb:main} holds with $\tau_k \equiv \tau$ and $1+\gamma_k\tau \ge \Lambda_k$ for all $k=0,\ldots,N-1$. Let $x^{1:N}$ generated by \cref{alg:fb:alg} for an initial $x^0 \in X$.
    Then
    \[
        \dynregret_{\PpredictConstr_{1:N}}(x^1,\ldots,x^N)
        +\sum_{k=0}^{N-1}
                \frac{1-\zeta_{k+1}}{2\tau}\norm{\nextx-A_k(\thisx)}^2
        \le
        \sup_{\optx^0 \in \PpredictConstr_0} \frac{\norm{x^0-\optx^0}^2}{2\tau\inv{(1+\gamma_0\tau)}}
        + \sum_{k=1}^N \frac{\epsilon_k}{\tau}.
    \]
\end{corollary}


\begin{remark}[Weighted dynamic regret]
    \label{rem:fb:weighted}
    Suppose $1+\gamma_k\tau_k > \Lambda_k$. Then $\{\tauTest_k\}_{k \in \N}$ can increase while satisfying $\tauTest_{k+1} \le \tauTest_k(1+\gamma_k\tau_k)\inv\Lambda_k$. If $\inf_k \tau_k  > 0$, then \eqref{eq:fb:estimate} places more importance on $J_k$ for large $k$: we regret early iterates less than recent.
    If $\tfrac{1+\gamma_k\tau_k}{\Lambda_k} \ge c > 1$ and $\tauTest_k = c^k\tauTest_0$, this growth in importance is exponential, comparable to linear convergence on static  problems; cf.~\cite{tuomov-proxtest}. With $F_{k+1} \equiv 0$ it is even possible to take $\tau_k \upto \infty$ and obtain superexponential growth (superlinear convergence).

    If, on the other hand $1+\gamma_k\tau_k < \Lambda_k$, then the condition $\tauTest_{k+1} \le \tauTest_k(1+\gamma_k\tau_k)\inv\Lambda_k$ forces $\{\tauTest_k\}_{k \in \N}$ to be decreasing. We therefore regret bad early iterates more than the recent. In the context of static optimisation problems, we are in the region of non-convergence or at most slow sub-$O(1/N)$ rates.
\end{remark}

\section{Partial gap functionals}
\label{sec:gap}

We start our development of a primal-dual method by deriving meaningful measures of regret. We cannot in general obtain estimates on conventional duality gaps or on iterates, so need  alternative criteria.
Throughout this section $F: X \to \extR$ and $G: Y \to \extR$ are convex, proper, and lower semicontinuous, and $K \in \linear(X; Y)$ on Hilbert spaces $X$ and $Y$.
We write $\GenLag(x, y) \defeq F(x) + \iprod{Kx}{y} - G^*(y)$ for the corresponding \term{Lagrangian}.
We recall that the first-order primal-dual optimality conditions for
\begin{gather}
    \nonumber
    \min_{x \in X} F(x)+G(Kx)
    \quad\text{equiv.}\quad
    \min_{x \in X} \max_{y \in Y}~ \GenLag(x, y)
\shortintertext{are}
    \label{eq:gap:oc}
    -K\realopty \in \subdiff F(\realoptx)
    \quad\text{and}\quad
    K^*\realoptx \in \subdiff G^*(\realopty).
\end{gather}
We call such a pair $(\realoptx, \realopty)$ a \term{critical point}.

\subsection{Common gap functionals}

By the Fenchel--Young inequality applied to $T(x, y) \defeq F(x)+G^*(y)$,  the \term{duality gap}
\[
    \GenGap(x, y) \defeq [F(x)+G(Kx)]+[F^*(-K^*y)+G^*(y)]  \ge 0,
\]
and is zero if and only if \eqref{eq:gap:oc} holds.
We can expand
\[
    \GenGap(x, y) = \sup_{(\optx,\opty) \in X \times Y}\left( \GenLag(x, \opty)-\GenLag(\optx, y)\right).
\]
This motivates the \term{Lagrangian duality gap}
\[
    \GenGap^\GenLag(x, y; \optx, \opty) \defeq \GenLag(x, \opty)-\GenLag(\optx, y).
\]
It is non-negative if $(\optx, \opty)$ is a critical point, but may be zero even if $(x, y)$ is not.

Since the Lagrangian duality gap is a relatively weak measure of optimality, and the true duality gap may not converge (fast), we define for bounded $B \subset X \times Y$ the \term{partial duality gap}
\[
    \GenGap_B(x, y) \defeq \sup_{(\optx, \opty) \in B} [\GenLag(x, \opty)-\GenLag(\optx, y)].
\]
This is non-negative if $B$ contains a critical point and equals the true duality gap $\GenGap$ if $B=X \times Y$. The partial gap converges ergodically for the basic unaccelerated PDPS \cite{chambolle2010first}.

\subsection{Partial primal gaps}

If we are not interested in the dual variable, we can define the \term{partial primal gap}
\begin{equation}
    \label{eq:gaps:hatgap-def}
    \hat\GenGap_B(x)
    \defeq \sup_{(\optx, \opty) \in B}  \inf_{y \in Y} \GenGap^\GenLag(x, y; \optx, \opty).
\end{equation}
We now try to interpret it.

\begin{lemma}
    \label{lemma:gaps:hatgap}
    Let $F: X \to \extR$ and $G: Y \to \extR$ be convex, proper, and lower semicontinuous, and $K \in \linear(X; Y)$. Pick $B \subset X \times Y$. Then
    \begin{align}
        \label{eq:gaps:hatgap-exp}
        \hat\GenGap_B(x) & = [F+\gapmod G \circ K](x)  - \inf_{(\optx, \opty) \in B} [F+G \circ K](\optx)
    \shortintertext{for}
        \label{eq:gaps:tildeg}
        \gapmod G(y') & \defeq \sup_{\optx \in X, \opty \in Y} \left(
            \iprod{y'}{\opty} - G^*(\opty) - J_B(\optx, \opty)
        \right) - J_{B}^*(0, 0)
        \quad\text{and}
        \\
        \nonumber
        J_B(\alt x, \alt y) & \defeq F(\alt x) + G(K\alt x) + \delta_{B}(\alt x, \alt y).
    \end{align}
\end{lemma}

\begin{proof}
    We have
    \[
        \begin{aligned}
        \inf_{y \in Y} \GenGap^\GenLag(x, y; \optx, \opty)
        &
        =\GenLag(x, \opty) - \sup_{y \in Y} \GenLag(\optx, y)
        \\
        &
        =F(x)+\iprod{Kx}{\opty}-G^*(\opty) - [F+G \circ K](\optx).
        \end{aligned}
    \]
    Thus
    \begin{equation*}
        \begin{aligned}[t]
        \hat\GenGap_B(x)
        &
        = F(x) + \sup_{\optx \in X, \opty \in Y} \left(
            \iprod{Kx}{\opty} - G^*(\opty) - [F + G \circ K](\optx) - \delta_B(\optx, \opty)
        \right)
        \\
        &
        = F(x) + \sup_{\optx \in X, \opty \in Y} \left(
            \iprod{Kx}{\opty} - G^*(\opty) - J_B(\optx, \opty)
        \right)
        =
        F(x) + \gapmod G(Kx) + J_B^*(0, 0).
        \end{aligned}
    \end{equation*}
    Since $J_{B}^*(0,0) = -\inf_{(\opt x, \opt y) \in B} [F + G \circ K](\opt x)$, this establishes the claim.
\end{proof}

\begin{example}
    If $B=B_X \times Y$ for some $B_X \subset X$, then $J_B(x,y)$ does not depend on $y$ so that we obtain $\gapmod G=G$.
    Thus the partial primal gap reduces to a standard difference of function values,
    \[
        \hat\GenGap_{B_X \times Y}(x) = [F+G \circ K](x) - \inf_{\optx \in B_X} [F+G \circ K](\optx).
    \]
    If now $B_X$ contains a minimiser of $F + G \circ K$, this difference is non-negative.
\end{example}

This example gives an indication towards the meaningfulness of the partial primal gap. In particular, if we take a smaller set $B$ than in the example, we can expect $\hat\GenGap_B(x)$ to attain smaller values. It may be negative even if $B_X$ contains a minimiser of $F+G \circ K$. This is akin to the regret functionals from the Introduction. Indeed,
we will use the partial primal gap as the basis for a \term{marginalised primal regret} that “fails to regret” what $F + \gapmod G \circ K \le F + G \circ K$ cannot measure.

In the applications of \cref{sec:flow}, $G(y^1,\ldots,y^N)=\sum_{k=1}^N \alpha \norm{Dy^k}_{\Meas}$, compare \eqref{eq:intro:rof}, and $B$ is a primal-dual extension $\PDpredictConstr_{1:N}$ of $\PpredictConstr_{1:N}$ from \eqref{eq:intro:PpredictConstr}. The construction of $\gapmod G$ convolves the static total variation regulariser $G$ with the temporally coupled objective $J_{\PDpredictConstr_{1:N}}$. The effect is to produce a new dynamic regulariser, alternative to \cite{iglesias2016convective,weickert2001variational,nagel1983constraints,nagel1990extending,chaudbury1995trajectory,salgado2007temporal,volz2011modeling}.
The following instructive proposition elucidates how this works in general. However, the convexity assumption on $B$ is not satisfied by $\PDpredictConstr_{1:N}$.
We write $E \infconv \alt E$ for the infimal convolution of $E, \alt E: X \to \extR$.

\begin{proposition}
    \label{lemma:gaps:tildeg-infconv}
    Suppose $B$ is closed, convex, and nonempty, and both $G$ and $J_B$ are coercive. Then
    \[
        \gapmod G(y')
        =
        \inf_{\alt y \in Y }\left( G(y'-\alt y) + J_{B}^*(0, \alt y) - J_{B}^*(0, 0)\right).
    \]
\end{proposition}

\begin{proof}
    We recall that $(E \infconv \alt E)^*=E^*+\alt E^*$ for proper $E, \alt E: X \to \extR$ \cite[Proposition 13.21]{bauschke2017convex}.
    The infimal convolution $E \infconv \alt E$ is convex, proper, and lower semicontinuous when $E$ and $\alt E$ also are, $E$ is coercive, and $\alt E$ is bounded from below \cite[Propositions 12.14]{bauschke2017convex}. Since then $(E \infconv \alt E)^{**}=E \infconv \alt E$, we obtain $E \infconv \alt E = (E^*+\alt E^*)^*$.

    By the convexity of $B$, $J_B=J_B^{**}$.
    The coercivity of $J_B$ implies that $J_B^*$ is bounded from below.\footnote{Any coercive, convex, proper, lower semicontinuous function $E: X \to \extR$ has a minimiser $\realoptx$. By the Fermat principle $0 \in \subdiff E(\realoptx)$.
    Thus $\realoptx \in \subdiff E^*(0)$, which says exactly that $E^* \ge E^*(0)$.}
    Since $G$ is coercive, taking $E(x,y)=G(y)+\delta_{\{0\}}(x)$ and $\alt E=J_B^*$, we get
    \[
        \begin{aligned}
        \gapmod G(y')
        &
        = \sup_{\alt x \in X, \alt y \in Y} \left(
            \iprod{y'}{\alt y} - G^*(\alt y) - J_{B}^{**}(\alt x, \alt y)
        \right) - J_{B}^*(0, 0)
        \\
        &
        =
        ([(\alt x, \alt y) \mapsto G^*(\alt y)]^* \infconv J_{B}^*)(0, y') - J_{B}^*(0, 0)
        \\
        &
        =
        ([(\alt x, \alt y) \mapsto G(\alt y) + \delta_{\{0\}}(\alt x)] \infconv J_{B}^*)(0, y') - J_{B}^*(0, 0)
        \\
        &
        =
        \inf_{\alt y \in Y }\left( G(y'-\alt y) + J_{B}^*(0, \alt y) - J_{B}^*(0, 0)\right).
        \qedhere
        \end{aligned}
    \]
\end{proof}

\begin{example}
    Take $B=B_X \times B_Y$ for some convex and closed $B_X \subset X$ and $B_Y \subset Y$. Then \cref{lemma:gaps:tildeg-infconv} gives $\gapmod G(y')= (G \infconv \delta_{B_Y}^*)(y')$.

    In particular, let $G=\alpha\norm{\freevar}_Y$ for some $\alpha>0$ and $B_Y=\B(\hat y, \rho) \defeq \{ y \in Y \mid \norm{y-\hat y }_Y \le \rho\}$ for some “expected solution” $\hat y$ and “confidence” $\rho>0$. Then $\gapmod G(y')=(G^*+\delta_{B_Y})^*(y')=\delta_{\B(0,\alpha) \isect \B(\hat y, \rho)}^*(y')$. If $\norm{\hat y}_Y=\alpha$ and $\rho < \alpha$, this means that $\gapmod G$ will not penalise points $y'=Kx$ with $\iprod{y'}{\hat y} \le 0$.
    We might interpret this as follows: since we are highly confident (small $\rho$) that $Kx \propto \hat y$ for an optimal $x$, we are not even interested in studying dual variables that point in the opposite direction.
    If $K$ were additionally a (discretised) gradient operator, as for total variation regularisation, roughly speaking this would say that we are not interested in studying gradients that point away from the expected gradient.
    
\end{example}

More generally, we can construct an infimal convolution lower bound with respect to the set of primal-dual minimisers of $J_B$. The coercivity assumption in the next lemma is fulfilled for $F$ the squared distance or $B$ bounded, both of which will be the case for the optical flow example.

\begin{proposition}
    Let $F: X \to \extR$ and $G: Y \to \extR$ be convex, proper, and lower semicontinuous, and $K \in \linear(X; Y)$. Pick a closed subset $B \subset X \times Y$ and suppose $J_B$ constructed from these components is coercive. Let
    \[
        \hat B \defeq \{(\optx,\opty) \in B \mid J_B(\optx,\opty)=\inf J_B\}
        \quad\text{and}\quad
        \hat B_Y \defeq \{\realopty \mid (\realoptx, \realopty) \in \hat B\}.
    \]
    Then $\gapmod G$ defined in \eqref{eq:gaps:tildeg} satisfies $\gapmod G \ge (G^*+\delta_{\hat B_Y})^*$.
\end{proposition}

\begin{proof}
    Since $J_B$ is coercive, lower semicontinuous, and bounded from below, $\hat B$ is non-empty. Since $\inf J_B=-J_B^*(0,0)$, we calculate
    \[
        \begin{aligned}
        \gapmod G(y') & \ge \sup_{(\realoptx,\realopty) \in \hat B} \left(
            \iprod{y'}{\realopty} - G^*(\realopty) - J_B(\realoptx, \realopty)
        \right) - J_{B}^*(0, 0)
        \\
        &
        = \sup_{\realopty \in \hat B_Y} \left(
            \iprod{y'}{\realopty} - G^*(\realopty)
        \right)
        = (G^*+\delta_{\hat B_Y})^*(y').
        \qedhere
        \end{aligned}
    \]
\end{proof}

\begin{remark}
    If $\hat B_Y$ is convex, then $\delta_{\hat B_Y}=\sigma_{\hat B_Y}^*$ for the support function $\sigma_{\hat B_Y}$. As this is convex, and lower semicontinuous, we get that $\gapmod G \ge G \infconv \sigma_{\hat B_Y}$.
\end{remark}

We always have $\gapmod G \le G$ since $- J_{B}^*(0, 0) \le J_{B}(\optx, \opty)$.
The following establishes a lower bound on $\gapmod G$ in the our typical case of interest, with $G$ a seminorm. It does not help interpret $\gapmod G$, but will be sufficient for developing regularisation theory in \cref{sec:flow}.

\begin{lemma}
    \label{lemma:gaps:indicator}
    Let $F: X \to \extR$ be convex, proper, and lower semicontinuous, and let $G=\delta_{B_Y}^*$ be the support function of a closed convex set $B_Y \subset Y$.  Pick $B \subset X \times B_Y$. Then $\gapmod G$ as defined in \eqref{eq:gaps:tildeg} satisfies $\gapmod G \ge -G(-\freevar)$.
\end{lemma}

\begin{proof}
    $(\optx,\opty) \in \Dom J$ implies $\opty \in B_Y$, hence $G^*(\opty)=\delta_{B_Y}(\opty)=0$. Thus
    \[
        \begin{aligned}
        \gapmod G(y') & = \sup_{\optx \in X, \opty \in Y} \left(
            \iprod{y'}{\opty} - J_B(\optx, \opty)
        \right) - J_{B}^*(0, 0)
        \\
        &
        \ge
        \inf_{(\optx, \opty) \in B} \iprod{y'}{\opty}
        +
        \sup_{\optx \in X, \opty \in Y} \left( - J_B(\optx, \opty) \right) - J_{B}^*(0, 0)
        \\
        &
        \ge
        \inf_{\opty \in B_Y} \iprod{y'}{\opty}
        =-\delta_{B_Y}^*(-y')=-G(-y').
        \qedhere
        \end{aligned}
    \]
\end{proof}

\section{Predictive online primal-dual proximal splitting}
\label{sec:pd}

We now develop for \eqref{eq:intro:problem-sequence} a predictive online version of the primal-dual proximal splitting (PDPS) of \cite{chambolle2010first}.
The structure is presented in \cref{alg:pd:alg}; our remaining work here consists of developing rules for the step length parameters $\tau_{k+1}$, $\sigma_{k+1}$, and $\tilde\sigma_{k+1}$ such that a low regret, for a suitable form of regret, is obtained.
\Cref{alg:pd:alg} consists of primal and dual steps (\cref{item:alg:pd:primal,item:alg:pd:dual}) that are analogous to the standard PDPS. Those are preceded by primal and dual prediction steps (\cref{item:alg:pd:primal-predict,item:alg:pd:dual-predict}). The primal prediction is basic, based on the user-prescribed predictor $A_k$, but the dual prediction is somewhat more involved, imposed by a our regret theory. In particular, it involves the somewhat arbitrary functions $\tilde G_{k+1}$.

\subsection{Assumptions and definitions}

To develop the regret theory, with the general notation $u=(x, y)$, $u^k=(x^k, y^k)$, etc., we work with the following setup:

\begin{algorithm}
    \caption{Predictive online primal-dual proximal splitting (POPD)}
    \label{alg:pd:alg}
    \begin{algorithmic}[1]
        \Require For all $k \in \N$, on Hilbert spaces $X_k$ and $Y_k$, convex, proper, lower semicontinuous $F_{k+1}: X_{k+1} \to \extR$ and $G_{k+1}^*, \tilde G_{k+1}^*: Y_{k+1} \to \extR$, predictors $A_k: X_k \to X_{k+1}$ and $B_k: Y_k \to Y_{k+1}$, and $K_{k+1} \in \linear(X_{k+1}; Y_{k+1})$.
        Step length parameters $\tau_{k+1},\sigma_{k+1},\tilde\sigma_{k+1}>0$.
        \State Pick initial iterates $x^0 \in X_0$ and $y^0 \in Y_0$.
        \For{$k \in \N$}
            \State\label{item:alg:pd:primal-predict} $\nexxt{\xi} \defeq A_k(\thisx)$
                \Comment{primal prediction}
            \State\label{item:alg:pd:dual-predict}$\nexxt{\upsilon} \defeq \prox_{\tilde\sigma_{k+1} \tilde G_{k+1}^*}(B_k(\thisy)+\tilde\sigma_{k+1}K_{k+1}\nexxt{\xi})$
                \Comment{dual prediction}
            \State\label{item:alg:pd:primal}$\nextx \defeq \prox_{\tau_{k+1} F_{k+1}}(\nexxt{\xi} - \tau_{k+1} K_{k+1}^*\nexxt{\upsilon})$
                \Comment{primal step}
            \State\label{item:alg:pd:dual} $\nexty \defeq \prox_{\sigma_{k+1} G_{k+1}^*}(\nexxt{\upsilon} + \sigma_{k+1} K_{k+1}(2\nextx-\nexxt{\xi}))$
                \Comment{dual step}
        \EndFor
    \end{algorithmic}
\end{algorithm}

\begin{assumption}
    \label{ass:pd:main}
    For all $k \ge 1$, on Hilbert spaces $X_k$ and $Y_k$, we assume to be given:
    \begin{enumerate}[label=(\roman*),nosep]
        \item convex, proper, and lower semicontinuous functions $F_k: X_k \to \extR$ and $G_k^*: Y_k \to \extR$, as well as $K_k \in \linear(X_k; Y_k)$.
        \item Primal and dual step length parameters $\tau_k,\sigma_k>0$.
        \item Primal and dual predictors $A_k: X_k \to X_{k+1}$ and $B_k: Y_k \to Y_{k+1}$.
        \item Some $\tilde\rho_{k+1}$-strongly convex, proper, and lower semicontinuous $\tilde G_{k+1}^*: Y_{k+1} \to \extR$ and parameters $\tilde\sigma_{k+1} > 0$.
    \end{enumerate}
    Further, we assume:
    \begin{enumerate}[resume*]
        \item\label{item:pd:main-assumption:predictconstr} to be given a bounded set of primal-dual comparison sequences
        \begin{gather*}%
            \!\!\!\!\!\!\!\PDpredictConstr \subset \left\{ \textstyle \optu^{0:\infty} \in \prod_{k=0}^\infty X_k \times Y_k \,\middle|\,
                \begin{array}{r}
                \opty^{k+1} = \prox_{\tilde\sigma_{k+1} \tilde G_{k+1}^*}(\nexxt{\tilde y}+\tilde\sigma_{k+1} K_{k+1} \nexxt\optx) \\ \text{for some } \nexxt{\tilde y} =: \nexxt{\tilde y}(\nexxt\optu) \in Y_{k+1},\, \forall k \ge 0 \end{array}\right\}
        \shortintertext{with which we define the set of primal comparison sequences as}
            \PpredictConstr \defeq \{ \optx^{0:\infty} \mid \optu^{0:\infty} \in \PDpredictConstr\}.
        \end{gather*}%
        \item\label{item:pd:main-assumption:predictbound} for some (Lipschitz-like) factors $\Lambda_k,\Theta_k > 0$ and \term{prediction penalties} $\epsilon_{k+1},\tilde\epsilon_{k+1} \in \R$ the primal and dual \term{prediction bounds}
            \begin{subequations}
            \label{eq:pd:prediction-bound}
            \begin{align}
                \label{eq:pd:primal-prediction-bound}
                \frac{1}{2}\norm{A_k(\thisx)-\nexxt\optx}_{X_{k+1}}^2 & \le \frac{\Lambda_k}{2}\norm{\thisx-\this\optx}_{X_k}^2 + \epsilon_{k+1}
                \quad\text{and}
                \\
                \label{eq:pd:dual-prediction-bound}
                \!\!\!\!\frac{1}{2}\norm{B_k(\thisy)-\nexxt{\tilde y}}_{Y_{k+1}}^2 & \le \frac{\Theta_k}{2}\norm{\thisy-\this\opty}_{Y_k}^2 + \tilde\epsilon_{k+1}
                \quad (\optu^{0:\infty} \in \PDpredictConstr,\, k \in \N),
            \end{align}
            \end{subequations}
            where $\PDpredictConstr$ and $\nexxt{\tilde y}$ are as in \cref{item:pd:main-assumption:predictconstr}, and $(\thisx,\thisy)$ are generated by \cref{alg:pd:alg}.
    \end{enumerate}
\end{assumption}

\begin{remark}
    \Cref{ass:pd:main}\,\cref{item:pd:main-assumption:predictconstr,item:pd:main-assumption:predictbound} are not directly needed for formulating \cref{alg:pd:alg}.
    They are needed to develop the regret theory.
    The Lipschitz-like constants $\Lambda_k$ and $\Theta_k$ will, however, appear in the step length rules that we develop.

    In a typical case $\nexxt{\optx} = \opt A_k(\this{\optx})$ and $\nexxt{\tilde y}=\opt B_k(\this{\opty})$ for some true (unknown) temporal coupling operators $\opt A_k$ and $\opt B_k$ that the (known) predictors $A_k$ and $B_k$ approximate. Then \eqref{eq:pd:prediction-bound} reads
    \begin{align*}
        \frac{1}{2}\norm{A_k(\thisx)-\opt A_k(\this{\optx})}_{X_{k+1}}^2 & \le \frac{\Lambda_k}{2}\norm{\thisx-\this\optx}_{X_k}^2 + \epsilon_{k+1}
        \quad\text{and}
        \\
        \frac{1}{2}\norm{B_k(\thisy)-\opt B_k(\this{\opty})}_{Y_{k+1}}^2 & \le \frac{\Theta_k}{2}\norm{\thisy-\this\opty}_{Y_k}^2 + \tilde\epsilon_{k+1}
        \quad (k \in \N)
    \end{align*}
    where the comparison points $\this{\optx}$ and $\this{\opty}$ are given through the recurrences $\nexxt{\optx}=A_k(\this{\optx})$ and $\opty^{k+1} = \prox_{\tilde\sigma_{k+1} \tilde G_{k+1}^*}(B_k(\this{\opty})+\tilde\sigma_{k+1} K_{k+1} \nexxt\optx)$.
    It may be easiest to omit the recurrences and prove the inequalities for any comparison points $\this{\optx}$ and $\this{\opty}$.
    If we had $A_k=\opt A_k$ and $B_k=\opt B_k$, and these operators were Lipschitz, we could take $\Lambda_k$ and $\Theta_k$ as the corresponding Lipschitz factors and the prediction errors $\epsilon_{k+1}=\tilde\epsilon_{k+1}=0$.
    Typically we would not know the true temporal coupling---or would know it only up to measurement noise---so need the prediction errors to model this lack of knowledge or noise.
\end{remark}

\begin{example}
    \label{ex:pd:standard-tilde-g}
    We can always take, and in practise take, $\tilde G_{k+1}^*=G_{k+1}^*+\tfrac{\tilde\rho_{k+1}}{2}\norm{\freevar}_{Y_{k+1}}^2$.
\end{example}

We now define for all $k \ge 1$ the monotone operator\footnote{The double arrow signifies that the map is set-valued.} $H_k: X_k \times Y_k \setto X_k \times Y_k$ and the linear preconditioned $M_k \in \linear(X_k \times Y_k; X_k \times Y_k)$ as
\begin{equation}
    \label{eq:pd:h-m}
    H_k(u) \defeq
    \begin{pmatrix}
        \subdiff F_k(x) + K_k^* y \\
        \subdiff G_k^*(y) - K_kx
    \end{pmatrix}
    \quad\text{and}\quad
    \Precond_k \defeq \begin{pmatrix}
        \inv\tau_k \Id & - K_k^* \\
        -K_k & \inv\sigma_k \Id
    \end{pmatrix}.
\end{equation}
Then $0 \in H_k(\this\realoptu)$ encodes the primal-dual optimality conditions \eqref{eq:gap:oc} for the static problem \eqref{eq:intro:minmax-sequence} while \cref{alg:pd:alg} can be written in implicit form as
\begin{equation}
    \label{eq:ppext-pdps}
    0 \in H_k(\thisu) + M_k(\thisu-\thisz)
    \quad (k \ge 1)
\end{equation}
for
\begin{equation}
    \label{eq:pd:sk}
    \nexxt z \defeq (\nexxt\xi, \nexxt\upsilon) \defeq S_k(\thisu),
    \ \ %
    S_k(u) \defeq
    \begin{pmatrix}
        A_k(x)
        \\
        \prox_{\tilde\sigma_{k+1} \tilde G_{k+1}^*}(B_k(y)+\tilde\sigma_{k+1} K_{k+1} A_k(x))
    \end{pmatrix}
    \quad (k \ge 0).
\end{equation}

We now derive regret estimates based on the partial primal gaps of \cref{sec:gap}.

\subsection{A general regret estimate}

We need the following \term{strong non-expansivity} from the dual predictor. The result is standard, but difficult to find explicitly stated in the literature for $\gamma>0$:

\begin{lemma}
    \label{lemma:pd:strong-non-expansivity}
    On a Hilbert space $X$, suppose $F: X \to \extR$ is convex, proper, and $\gamma$-strongly subdifferentiable. Then $\prox_F$ is $(1+\gamma)$-strongly non-expansive:
    \[
        (1+\gamma)\norm{\prox_F(x)-\prox_F(\alt x)}_X^2 \le \iprod{\prox_F(x)-\prox_F(\alt x)}{x-\alt x}
        \quad (x, \alt x \in X).
    \]
\end{lemma}

\begin{proof}
    Let $y \defeq \prox_F(x)$.
    By definition, $y+q=x$ and $\alt y+\alt q=\alt x$ for some $q \in \subdiff F(y)$ and $\alt q \in \subdiff F(\alt y)$. Since $\subdiff F$ is $\gamma$-strongly monotone, $\iprod{q-\alt q}{y-\alt y} \ge \gamma\norm{y-\alt y}^2$.
    Thus
    \[
        (1+\gamma)\norm{y-\alt y}^2
        =\iprod{y-\alt y}{x-\alt x-(q-\alt q)}+\gamma\norm{y-\alt y}^2
        \le \iprod{y-\alt y}{x-\alt x}.
        \qedhere
    \]
\end{proof}

The next lemma derives basic step length conditions, which we will further develop in \cref{sec:pd:steps}, from basic properties of the linear preconditioner $M_k$ and an overall primal-dual prediction bound analogous to \cref{eq:pd:prediction-bound}.
The “testing” parameters $\tauTest_k,\sigmaTest_k,\eta_k > 0$ model the respective primal, dual, and joint (e.g., gap) convergence or regret rates. They are coupled via \eqref{eq:pd:symcond} to the step length parameters. Any one of these parameters is superfluous given the others, but all are included for notational and conceptual convenience. The testing parameters are not directly required in \cref{alg:fb:alg}, but will serve to study “regret rates”.

\begin{lemma}
    \label{lemma:pd:stepconds}
    Suppose \cref{ass:pd:main} holds. Fix $k \in \N$ and assume for some $\kappa \in (0, 1)$ and \term{testing parameters} $\eta_k,\tauTest_k,\sigmaTest_k > 0$, the step length conditions
    \begin{subequations}%
    \label{eq:pd:stepconds1}%
    \begin{align}%
        \label{eq:pd:symcond}
        \eta_k & = \tauTest_k\tau_k = \sigmaTest_k \sigma_k,
        && \text{(primal-dual coupling)}
        \\
        \label{eq:pd:dualtestcond}
        \tilde\rho_{k+1}
        & \ge
        \frac{\Theta_k\eta_{k+1}\tilde\sigma_{k+1}^{-2}}{2\kappa(1+\sigma_k\rho_k)\sigmaTest_k}
        +\frac{1}{2\sigma_{k+1}}
        -\inv{\tilde\sigma_{k+1}},
         && \text{(proximal predictor restriction)}
        \\
        \label{eq:pd:primaltestcond}
        \tauTest_k(1+\gamma_k\tau_k)
        & \ge
        \tauTest_{k+1}\Lambda_k
        +\frac{\tauTest_k\tau_k\sigma_k\norm{K_k}^2}{(1-\kappa)(1+\sigma_k\rho_k)},
        && \text{(primal metric update)}
        \quad\text{and}
        \\
        \label{eq:pd:primaltestcond2}
        1
        & \ge
        \tau_k\sigma_k\norm{K_k}^2
        && \text{(metric positivity).}
    \end{align}%
    \end{subequations}%
    Let
    \begin{equation}
        \label{eq:pd:gamma}
        \Gamma_k \defeq
        \eta_k
        \begin{pmatrix}
            \gamma_k  \Id & 2K_k^* \\
            -2 K_k & \rho_k \Id
        \end{pmatrix}.
    \end{equation}
    Then $\eta_k\Precond_k$ is self-adjoint and positive semidefinite, $\eta_k\Precond_{k}+\Gamma_{k}$ is positive semidefinite, and we have the overall prediction bound
    \begin{equation}
        \label{eq:pdps:prediction-bound}
        \frac{1}{2}\norm{\nextz-\nexxt\optx}_{\eta_{k+1}\Precond_{k+1}}^2 \le \frac{1}{2}\norm{\thisx-\this\optx}_{\eta_k\Precond_k+\Gamma_k}^2 + \tauTest_{k+1}\epsilon_{k+1}
        + \frac{\kappa(1+\sigma_k\rho_k)\sigmaTest_k}{2\Theta_k} \tilde\epsilon_{k+1}
        \quad (k=0,\ldots,N-1).
    \end{equation}
\end{lemma}

\begin{proof}
    Using \eqref{eq:pd:symcond} and Young's inequality, we expand and estimate
    \begin{equation}
        \label{eq:linpred:pd:zm}
        \eta_k\Precond_k
        =
        \begin{pmatrix}
            \tauTest_k \Id & -\eta_k K_k^* \\
            -\eta_k K_k & \sigmaTest_k \Id
        \end{pmatrix}
        \ge
        \begin{pmatrix}
            \tauTest_k \Id - \eta_k^2\inv\sigmaTest_{k} K_k^*K_k & 0 \\
            0 & 0
        \end{pmatrix}.
    \end{equation}
    Thus $\eta_k\Precond_k$ is self-adjoint due to \eqref{eq:pd:symcond} and positive semidefinite due to \eqref{eq:pd:primaltestcond2} and \eqref{eq:pd:symcond}.
    It follows, using Young's inequality, that
    \begin{multline}
        \label{eq:pd:testprecond+gamma}
        \eta_{k}\Precond_{k}+\Gamma_{k}
        \simeq
        \begin{pmatrix}
            \tauTest_k(1+\gamma_{k}\tau_k) \Id & -\eta_k K_k^*\\
            -\eta_k K_k & \sigmaTest_k(1+\rho_{k}\sigma_k) \Id
        \end{pmatrix}
        \ge
        \begin{pmatrix}
            \tauTest_k(1+\gamma_k\tau_k) \Id - \frac{\eta_k^2}{\sigmaTest_{k}(1+\rho_k\sigma_k)} K_k^*K_k & 0 \\
            0 & 0
        \end{pmatrix}.
    \end{multline}
    Thus $\eta_{k}\Precond_{k}+\Gamma_{k}$ is positive semidefinite by \eqref{eq:pd:primaltestcond} and \eqref{eq:pd:symcond}.

    We still need to prove \eqref{eq:pdps:prediction-bound}.
    Writing $(\nexxt\xi, \nexxt\upsilon) \defeq \nextz =  S_k(\thisu)$, we have
    \begin{equation}
        \label{eq:pd:predictineq0}
        \begin{aligned}[t]
            \frac{1}{2}\norm{\nextz-\nexxt\optu}^2_{\eta_{k+1}\Precond_{k+1}}
            & =
            \frac{\tauTest_{k+1}}{2}\norm{A_k(\thisx)-\nexxt\optx}^2
            +
            \frac{\sigmaTest_{k+1}}{2}\norm{\nexxt\upsilon-\nexxt\opty}^2
            \\
            \MoveEqLeft[-1]
            -\eta_{k+1} \iprod{K_{k+1}(\nexxt\xi-\nexxt\optx)}{\nexxt\upsilon-\nexxt\opty}
        \end{aligned}
    \end{equation}
    as well as
    \begin{equation}
        \label{eq:pd:predictineq1}
        \begin{aligned}[t]
        \frac{1}{2}\norm{\thisu-\this\optu}^2_{\eta_k\Precond_k+\Gamma_k}
        & =
        \frac{\tauTest_k(1+\gamma_k\tau_k)}{2}\norm{\thisx-\this\optx}^2
        +
        \frac{\sigmaTest_k(1+\rho_k\sigma_k)}{2}\norm{\thisy-\this\opty}^2
        \\
        \MoveEqLeft[-1]
        -\eta_k \iprod{K_k(\thisx-\this\optx)}{\thisy-\this\opty}.
        \end{aligned}
    \end{equation}

    Since $\tilde G_{k+1}$ is ($\tilde\rho_{k+1}$-strongly) convex, by \cref{lemma:pd:strong-non-expansivity},  \eqref{eq:pd:sk}, and \cref{ass:pd:main},\,\cref{item:pd:main-assumption:predictconstr}
    \[
        (1+ \tilde\sigma_{k+1}\tilde\rho_{k+1})\norm{\nexxt\upsilon-\nexxt\opty}^2
        \le
        \iprod{\nexxt\upsilon-\nexxt\opty}{B_k(\thisy)-\nexxt{\tilde y}+\tilde\sigma_{k+1} K_{k+1}(\nexxt\xi-\nexxt\optx)}.
    \]
    By \eqref{eq:pd:symcond} and \eqref{eq:pd:dualtestcond},
    \[
        -\eta_{k+1}\inv{\tilde\sigma_{k+1}}(1+ \tilde\sigma_{k+1}\tilde\rho_{k+1})
        + \frac{\Theta_k\eta_{k+1}^2\tilde\sigma_{k+1}^{-2}}{2\kappa(1+\sigma_k\rho_k)\sigmaTest_k}
        \le -\frac{\sigmaTest_{k+1}}{2}.
    \]
    Consequently, also using \eqref{eq:pd:symcond}, \eqref{eq:pd:dual-prediction-bound}, and Young's inequality, we obtain
    \begin{equation*}
        \begin{aligned}[t]
        -\eta_{k+1} & \iprod{K_{k+1}(\nexxt\xi-\nexxt\optx)}{\nexxt\upsilon-\nexxt\opty}
        \\
        &
        =
        -\eta_{k+1}\inv{\tilde\sigma_{k+1}} \iprod{B_k(\thisy)-\nexxt{\tilde y}+\tilde\sigma_{k+1} K_{k+1}(\nexxt\xi-\nexxt\optx)}{\nexxt\upsilon-\nexxt\opty}
        \\
        \MoveEqLeft[-1]
        +\eta_{k+1}\inv{\tilde\sigma_{k+1}} \iprod{B_k(\thisy)-\nexxt{\tilde y}}{\nexxt\upsilon-\nexxt\opty}
        \\
        &
        \le
        -\eta_{k+1}\inv{\tilde\sigma_{k+1}}(1+ \tilde\sigma_{k+1}\tilde\rho_{k+1})\norm{\nexxt\upsilon-\nexxt\opty}^2
        +\eta_{k+1}\inv{\tilde\sigma_{k+1}}\iprod{B_k(\thisy)-\nexxt{\tilde y}}{\nexxt\upsilon-\nexxt\opty}
        \\
        &
        \le
        -\frac{\sigmaTest_{k+1}}{2}\norm{\nexxt\upsilon-\nexxt\opty}^2
        +
        \frac{\kappa(1+\sigma_k\rho_k)\sigmaTest_k}{2\Theta_k}\norm{B_k(\thisy)-\nexxt{\tilde y}}^2.
        \\
        &
        \le
        -\frac{\sigmaTest_{k+1}}{2}\norm{\nexxt\upsilon-\nexxt\opty}^2
        +
        \frac{\kappa(1+\sigma_k\rho_k)\sigmaTest_k}{2\Theta_k}
        \left(\frac{\Theta_k}{2}\norm{\thisy-\this\opty}^2 + \tilde\epsilon_{k+1}\right).
        \end{aligned}
    \end{equation*}
    Applying this and \eqref{eq:pd:primal-prediction-bound} in  \eqref{eq:pd:predictineq0}, we obtain for $p_{k+1} \defeq \tauTest_{k+1}\epsilon_{k+1} + \frac{\kappa(1+\sigma_k\rho_k)\sigmaTest_k}{2\Theta_k} \tilde\epsilon_{k+1}$ that
    \begin{equation}
        \label{eq:pd:predictineq0-2}
        \frac{1}{2}\norm{\nextz-\nexxt\optu}^2_{\eta_{k+1}\Precond_{k+1}}
        \le
        \frac{\tauTest_{k+1}\Lambda_k}{2}\norm{\thisx-\this\optx}^2
        +\frac{\kappa(1+\sigma_k\rho_k)\sigmaTest_k}{2}\norm{\thisy-\this\opty}^2
        + \PredictPenalty_{k+1}.
    \end{equation}

    We also have by Young's inequality
    \[
        \begin{aligned}[t]
        \eta_k \iprod{K_k&(\thisx-\this\optx)}{\thisy-\this\opty}
        \\
        &
        \le \frac{\eta_k^2}{2(1-\kappa)(1+\sigma_k\rho_k)\sigmaTest_k}\norm{K_k(\thisx-\this\optx)}^2
        +\frac{(1-\kappa)(1+\sigma_k\rho_k)\sigmaTest_k}{2}\norm{\thisy-\this\opty}^2.
        \end{aligned}
    \]
    Hence \eqref{eq:pd:predictineq1} gives
    \begin{equation}
        \label{eq:pd:predictineq1-2}
        \begin{aligned}[t]
        -\frac{1}{2}\norm{\thisu-\this\optu}^2_{\eta_k\Precond_k+\Gamma_k}
        & \le
        \left(
            \frac{\eta_k^2\norm{K_k}^2}{2(1-\kappa)(1+\sigma_k\rho_k)\sigmaTest_k}
            -\frac{\tauTest_k(1+\gamma_k\tau_k)}{2}
        \right)\norm{\thisx-\this\optx}^2
        \\
        \MoveEqLeft[-1]
        - \frac{\kappa(1+\sigma_k\rho_k)\sigmaTest_k}{2} \norm{\thisy-\this\opty}^2.
        \end{aligned}
    \end{equation}

    Combined, \eqref{eq:pd:predictineq0-2} and \eqref{eq:pd:predictineq1-2} show that
    \[
        \begin{aligned}[t]
        \frac{1}{2}\norm{\nextz-\nexxt\optu}^2_{\eta_{k+1}\Precond_{k+1}}
        &-\frac{1}{2}\norm{\thisu-\this\optu}^2_{\eta_k\Precond_k+\Gamma_k}
        \le
        \tauTest_{k+1}\epsilon_{k+1}
        + \PredictPenalty_{k+1}
        \\
        \MoveEqLeft[-1]
        +\left(
            \frac{\tauTest_{k+1}\Lambda_k}{2}
            +\frac{\eta_k^2\norm{K_k}^2}{2(1-\kappa)(1+\sigma_k\rho_k)\sigmaTest_k}
            -\frac{\tauTest_k(1+\gamma_k\tau_k)}{2}
        \right)\norm{\thisx-\this\optx}^2.
        \end{aligned}
    \]
    From here \eqref{eq:pd:primaltestcond} shows \eqref{eq:pdps:prediction-bound}.
\end{proof}

To state the final regret estimate, for brevity we define
\begin{gather*}
    F_{1:N}(x^{1:N}) \defeq \sum_{k=0}^{N-1} \eta_k F_{k+1}(\nextx),
    \quad
    G_{1:N}(y^{1:N}) \defeq \sum_{k=0}^{N-1} \eta_k G_{k+1}(\inv\eta_k \nexty),
    \quad\text{and}
    \\
    K_{1:N}x^{1:N} \defeq (\eta_0 K_{1} x^1, \ldots, \eta_{N-1} K_N x^N).
\end{gather*}

We recall the comparison sets $\PDpredictConstr$ and $\PpredictConstr$ and from \cref{ass:pd:main} and the slicing notation $\PDpredictConstr_{n:m}$ and $\PpredictConstr_{n:m}$ form \cref{sec:intro}. With these we also define
\begin{equation}
    \label{eq:pd:tildeg}
    \gapmod G_{1:N}(y'_{1:N})
    =
    \sup_{\optx^{1:N}, \opty^{1:N}} \left(
        \iprod{y'_{1:N}}{\opty^{1:N}} - G_{1:N}^*(\opty^{1:N})
        -  J_{\PDpredictConstr_{1:N}}(\optx^{1:N}, \opty^{1:N})
    \right)
    - J_{\PDpredictConstr_{1:N}}^*(0, 0)
\end{equation}
with the supremum running over $\optx^{1:N} \in X_1 \times \cdots \times X_N$ and $\opty^{1:N} \in Y_1 \times \cdots \times Y_N$ and
\[
    J_{\PDpredictConstr_{1:N}}(\alt x^{1:N}, \alt y^{1:N}) \defeq [F_{1:N}+G_{1:N} \circ K_{1:N}](\alt x^{1:N})+\delta_{\PDpredictConstr_{1:N}}(\alt x^{1:N},\alt y^{1:N}).
\]
Observe that $G_{1:N}^*(y^{1:N})=\sum_{k=0}^{N-1} \eta_k G_{k+1}^*(\nexty)$ and
\begin{equation}
    \label{eq:pd:multi-fv}
    [F_{1:N}+G_{1:N} \circ K_{1:N}](x^{1:N})
    =\sum_{k=0}^{N-1} \eta_k[F_{k+1} + G_{k+1} \circ K_{k+1}](\nextx).
\end{equation}

After the next main regret estimate, we comment upon its assumptions and claim.

\begin{theorem}
    \label{thm:pd:main}
    Suppose \cref{ass:pd:main} and the step length bounds \eqref{eq:pd:stepconds1} hold for $u^{1:N}$ generated by \cref{alg:pd:alg} for an initial $u^0 \in X_0 \times Y_0$. Then
    \begin{multline*}
        [F_{1:N}+\gapmod G_{1:N} \circ K_{1:N}](x^{1:N}) - \inf_{\optx^{1:N}\in \PpredictConstr_{1:N}} [F_{1:N} + G_{1:N} \circ K_{1:N}](\optx^{1:N})
        +\sum_{k=0}^{N-1} \frac{\norm{\nextu-S_k(\thisu)}_{\eta_{k+1}\Precond_{k+1}}^2}{2}
        \\
        \le
        e_N \defeq
        \sup_{\optu^0 \in \PDpredictConstr_0} \frac{1}{2}\fakenorm{u^0-\optu^0}^2_{\eta_0\Precond_0+\Gamma_0}
        + \sum_{k=0}^{N-1}\left(
            \epsilon_{k+1}\tauTest_{k+1}
            +\frac{\kappa(1+\sigma_k\rho_k)\sigmaTest_k}{2\Theta_k} \tilde\epsilon_{k+1}
        \right).
    \end{multline*}
\end{theorem}

\begin{proof}
    For brevity, and to not abuse norm notation when $\Gamma_k$ is not positive semi-definite, we write $\fakenorm{x}_{\Gamma_k}^2 \defeq \iprod{x}{x}_{\Gamma_k}$.
    By \cref{lemma:pd:stepconds}, $\eta_k\Precond_k$ and  $\eta_k\Precond_k+\Gamma_k$ are positive semi-definite, so we may use the norm notation with them.
    For $H_k$ defined \eqref{eq:pd:h-m} and $\Gamma_k$ and $\eta_k$ in \eqref{eq:pd:gamma}, the (strong) convexity of $F_k$ and $G_k^*$ yield
    \begin{gather}
        \label{eq:pdps:condition}
        \begin{aligned}[t]
        \iprod{H_k(\thisu)}{\thisu-\this\optu}_{\eta_k}
        &
        \ge
        \frac{1}{2}\fakenorm{\thisu-\this\optu}_{\Gamma_k}^2
        + \GenGap^H_k
        \quad (k=1,\ldots,N)
        \end{aligned}
    \shortintertext{for}
        \label{eq:pd:gengap}
        \begin{aligned}[t]
        \GenGap_{k+1}^H
        &
        \defeq
        \eta_k[F_{k+1}(\nextx)-F_{k+1}(\nexxt\optu)+G_{k+1}^*(\nexty)-G_{k+1}^*(\nexxt\opty)
        \\
        \MoveEqLeft[-3]
        -\iprod{K_{k+1}^*\nexty}{\nexxt\optu}
        +\iprod{K_{k+1}\nextx}{\nexxt\opty}].
        \end{aligned}
    \end{gather}
    Following the testing methodology of \cite{tuomov-proxtest,clasonvalkonen2020nonsmooth}, we pick any $\this\optu \in X_k \times Y_k$ and apply the linear “testing operator” $\iprod{\freevar}{\thisu-\this{\optu}}_{\eta_k}$ to both sides of \eqref{eq:ppext-pdps}. This followed by \eqref{eq:pdps:condition} yields
    \[
        0 \ge
        \iprod{\thisu-\this z}{\thisu-\this{\optu}}_{\eta_k\Precond_k}
        + \frac{1}{2}\fakenorm{\thisu-\this\optu}_{\Gamma_k}^2
        + \GenGap^H_k
        \quad (k=1,\ldots,N).
    \]
    Pythagoras' identity \eqref{eq:fb:pythagoras} for the inner product and norm with respect to the operator $\eta_k\Precond_k$ now yields
    \[
        \frac{1}{2}\norm{\thisz-\this{\optu}}_{\eta_k \Precond_k}^2
        \ge
        \frac{1}{2}\norm{\thisu-\this{\optu}}_{\eta_k \Precond_k+\Gamma_k}^2
        + \GenGap^H_k
        + \frac{1}{2}\norm{\thisu-\this z}_{\eta_k\Precond_k}^2
        \quad (k=1,\ldots,N).
    \]
    We now take $\optu^{0:N} \in \PDpredictConstr_{0:N}$ and apply the prediction bound \eqref{eq:pdps:prediction-bound} from \cref{lemma:pd:stepconds} to obtain
    \begin{multline*}
        \frac{1}{2}\norm{\thisu-\this{\optu}}_{\eta_k \Precond_k+\Gamma_k}^2
        + \left(\tauTest_{k+1}\epsilon_{k+1}
        + \frac{\kappa(1+\sigma_k\rho_k)\sigmaTest_k}{2\Theta_k} \tilde\epsilon_{k+1}\right)
        \\
        \ge
        \frac{1}{2}\norm{\nextu-\nexxt{\optu}}_{\eta_{k+1} \Precond_{k+1}+\Gamma_{k+1}}^2
        + \GenGap^H_{k+1}
        + \frac{1}{2}\norm{\thisu-\this z}_{\eta_k\Precond_k}^2
        \quad (k=1,\ldots,N-1).
    \end{multline*}
    Summing over such $k$ and taking the supremum over $\optu^{0:N}\in \PDpredictConstr_{0:N}$, we get
    \[
        \sup_{\optu^{0:N}\in \PDpredictConstr_{0:N}}
        \sum_{k=0}^{N-1}\left( \GenGap^H_{k+1} + \frac{1}{2}\norm{\nextu-\nexxt z}_{\eta_{k+1}\Precond_{k+1}}^2\right)
        \le e_N.
    \]
    By \cref{lemma:gaps:hatgap} applied to $K=K_{1:N}$, $F=F_{1:N}$ and $G^*=G_{1:N}^*$ and \eqref{eq:pd:gengap} we obtain
    \[
        \sup_{\optu^{1:N} \in \PDpredictConstr_{1:N}}
        \sum_{k=0}^{N-1} \GenGap_{k+1}^H
        \ge
        [F_{1:N}+\gapmod G_{1:N} \circ K_{1:N}](x^{1:N})
        - \inf_{\optx^{1:N}\in \PpredictConstr_{1:N}} [F_{1:N} + G_{1:N} \circ K_{1:N}](\optx^{1:N}).
    \]
    Since $\nexxt z \defeq S_k(\thisu)$ by \eqref{eq:pd:sk}, these two inequalities together verify the claim.
\end{proof}

\begin{remark}[Satisfying the conditions]
    \Cref{ass:pd:main} is structural. Aside from $\tilde G_{k+1}$, everything in it depends on the application problem and the predictors we can design for it. The function $\tilde G_{k+1}$ can be taken as in \cref{ex:pd:standard-tilde-g}.
    The step length bounds \eqref{eq:pd:stepconds1} can be satisfied via the choices in the next \cref{sec:pd:steps}.
\end{remark}

\begin{remark}[Interpretation of the dual comparison sequence]
    Let $\nexxt{\tilde y}=\opt B_k(\this\opty)$ for a dual temporal coupling operator $\opt B_k$.
    Then the definition of $\PDpredictConstr$ in \cref{ass:pd:main}\,\cref{item:pd:main-assumption:predictconstr} updates the dual comparison variable as
    \begin{equation}
        \label{eq:pd:dualcoupling-online}
        \nexxt\opty \defeq \prox_{\tilde\sigma_{k+1} \tilde G_{k+1}^*}(\opt B_k(\this\opty)+\tilde\sigma K_{k+1}\nexxt\optx)
    \end{equation}
    This amounts to the POFB of \cref{sec:fb} applied with the predictor $\opt B_k$ and the step length parameter $\tau_{k+1}=\tilde\sigma_{k+1}$ to the formal problem
    \begin{equation}
        \nonumber
        \min_{y^1,y^2,\ldots} \sum_{k=1}^\infty \tilde G_k^*(y^k)-\iprod{K_k\this\optx}{y^k},
        \quad y^{k+1}=\opt B_k(y^k)
    \end{equation}
    An “optimal” $\realopty^k$, achieving $\inf_y \tilde G_k^*(y)-\iprod{K_k\this\optx}{y}$, would give
    \[
        [F_k+\tilde G_k \circ K_k](\this\optx)
        =
        F_k(\this\optx)+\iprod{K_k\this\optx}{\this\realopty} - \tilde G_k^*(\this\realopty).
    \]
    This is approximated by $\nexxt\opty$ generated by \eqref{eq:pd:dualcoupling-online}, better as $\tilde\sigma_k \upto \infty$. In the setting of \cref{ex:pd:standard-tilde-g}, if also $\tilde\rho_k \downto 0$, then we get closer to calculating $[F_k+G_k \circ K_k](\this\optx)$.
\end{remark}

\subsection{Specific step length choices}
\label{sec:pd:steps}

We now develop explicit step length rules that satisfy the step length conditions \eqref{eq:pd:stepconds1}, and then interpret \cref{thm:pd:main} for them.
The proof of the next lemma is immediate:

\begin{lemma}
    \label{lemma:pd:cond-simpl}
    The right hand side of \eqref{eq:pd:dualtestcond} is minimised by $\tilde\sigma_{k+1}= \tfrac{\Theta_k\eta_{k+1}}{\kappa(1+\sigma_k\rho_k)\sigmaTest_k}$. With this choice \eqref{eq:pd:dualtestcond} reads
    $
    2\eta_{k+1}\tilde\rho_{k+1} \ge \sigmaTest_{k+1}-\kappa\inv\Theta_k(1+\sigma_k\rho_k)\sigmaTest_k.
    $
\end{lemma}

The following examples use \cref{lemma:pd:cond-simpl}:

\begin{example}[Constant step length and testing parameters]
    \label{ex:pd:stepconds:const}
    In \cref{alg:pd:alg}, take as the step length parameters $\tau_k \equiv \tau$, $\sigma_k \equiv \sigma$, and $\tilde\sigma_{k+1}=\tfrac{\Theta_k\sigma}{\kappa(1+\sigma\rho_k)}$ for some constant $\tau,\sigma>0$ and $\kappa \in (0, 1)$ satisfying for the strong convexity factors $\gamma_k,\rho_k, \tilde\rho_{k+1}$ and the Lipschitz-like factors $\Theta_k$, $\Lambda_k$ from \cref{ass:pd:main} the inequalities
    \begin{equation}
        \label{eq:pd:stepconst:const:cond}
        \tilde\rho_{k+1}
        \ge
        \frac{1}{2\sigma}\left(1
        -\frac{\kappa(1+\sigma\rho_k)}{\Theta_k}\right),
        \quad
        1+\gamma_k\tau
        \ge
        \Lambda_k
        +\frac{\tau\sigma\norm{K_k}^2}{(1-\kappa)(1+\sigma\rho_k)},
        \quad\text{and}\quad
        1 \ge \tau\sigma\norm{K_k}^2.
    \end{equation}
    (By \cref{ex:pd:standard-tilde-g}, we may simply define $\tilde\rho_{k+1}$ through the first expression if we choose to take $\tilde G_{k+1}^*=G_{k+1}^*+\tfrac{\tilde\rho_{k+1}}{2}\norm{\freevar}_{Y_{k+1}}^2$.)
    Then \eqref{eq:pd:stepconds1} holds for the testing parameters $\eta_k \equiv \tau$, $\tauTest_k \equiv 1$, and $\sigmaTest_k \equiv \tfrac{\tau}{\sigma}$.
    In this case, \cref{thm:pd:main} shows for an initialisation-dependent constant $C_0$ that
    \[
        [F_{1:N}+\gapmod G_{1:N} \circ K_{1:N}](x^{1:N}) - \inf_{\optx^{1:N}\in \PpredictConstr_{1:N}} [F_{1:N} + G_{1:N} \circ K_{1:N}](\optx^{1:N})
        \le
        C_0
        + \sum_{k=0}^{N-1}\left(
            \epsilon_{k+1}
            +\frac{\kappa(1+\sigma\rho_k)\tau}{2\Theta_k\sigma} \tilde\epsilon_{k+1}
        \right).
    \]
    Suppose $\sup_k \rho_k \le \MAX \rho$ and $\inf_k \Theta_k \ge \MIN\Theta$ for some $\MAX\rho,\MIN\Theta>0$ (such as when $\rho_k$ and $\Theta_k$ are constant in $k$). Minding the sum expression \eqref{eq:pd:multi-fv}, where now $\eta_k=\tau$, for a constant $C>0$, we get
    \[
        \frac{1}{N}[F_{1:N}+\gapmod G_{1:N} \circ K_{1:N}](x^{1:N})
        -\inf_{\optx^{1:N}\in \PpredictConstr_{1:N}} \frac{\tau}{N}\sum_{k=0}^{N-1} [F_{k+1} + G_{k+1} \circ K_{k+1}](\nexxt{\optx})
        \le
        \frac{C_0+C\sum_{k=0}^{N-1} (\epsilon_{k+1}+\tilde\epsilon_{k+1})}{N}.
    \]
    Exact interpretation requires being able to calculate $\gapmod G_{1:N}$, however we can make a rough interpretation. We distinguish two cases:
    \begin{enumerate}[label=(\alph*)]
        \item
        If $\sum_{k=0}^\infty (\epsilon_{k+1}+\tilde\epsilon_{k+1})<\infty$, then the left hand side converges below zero as $N \to \infty$. Roughly, subject to how well we can measure with $\gapmod G_{1:N}$ in place of $G_{1:N}$, this says that asymptotically $x^{1:N}$ are at least as good solutions of the averaged problem $\inf_{x^{1:N}} \frac{\tau}{N}\sum_{k=0}^{N-1} [F_{k+1} + G_{k+1} \circ K_{k+1}](\nextx)$ as the best constrained $\optx^{1:N} \in \PpredictConstr_{1:N}$.
        \item
        If $\frac{1}{N}\sum_{k=0}^{N-1} (\epsilon_{k+1}+\tilde\epsilon_{k+1}) \le \delta$ for some constant $\delta>0$, then, again subject to how well we can measure with $\gapmod G_{1:N}$ in place of $G_{1:N}$, this says that asymptotically $x^{1:N}$ stays “within average noise level” $\delta$ of the best $\optx^{1:N} \in \PpredictConstr_{1:N}$.
    \end{enumerate}
    The bounds on the the prediction errors $\epsilon_{k+1}$ and $\tilde\epsilon_{k+1}$ can be interpreted as the noise level of the “measurements” $A_k$ and $B_k$ of the true temporal coupling operators $\opt A_k$ and $\opt B_k$ either vanishing or staying bounded (on average). In the optical flow example, to be further studied in \cref{sec:flow}, this means that the noise level of the displacement field measurements has to vanish or stay bounded (on average).
\end{example}

\begin{example}[Everything constant]
    \label{ex:pd:stepconds:constconst}
    In particular, in \cref{ex:pd:stepconds:const}, if the strong convexity and Lipshitz-like parameters are constant, $\tilde\rho_{k+1} \equiv \tilde\rho$, $\gamma_k \equiv \gamma$, and $\Theta_k \equiv \Theta$, and $\Lambda_k \equiv \Lambda$, with no dual strong convexity, $\rho_k=0$, and  we take $\tilde\sigma_{k+1} \equiv \tilde\sigma = \tfrac{\Theta\sigma}{\kappa}$, then \eqref{eq:pd:stepconst:const:cond}, hence \eqref{eq:pd:stepconds1}, hold if
    \[
        \tilde\rho
        \ge
        \frac{1-\kappa\inv\Theta}{2\sigma},
        \quad
        1+\gamma\tau
        \ge
        \Lambda
        +\frac{\tau\sigma\norm{K_k}^2}{1-\kappa},
        \quad\text{and}\quad
        1 \ge \tau\sigma\norm{K_k}^2.
    \]
\end{example}

\Cref{ex:pd:stepconds:const,ex:pd:stepconds:constconst} give no growth for the testing parameters $\tauTest_k, \sigmaTest_k$, and $\eta_k$. We now look at one case when this is possible and what happens then.

\begin{example}[Exponential testing parameters with constant step lengths]
    \label{ex:pd:stepconds:constexp}
    In \cref{alg:pd:alg}, take $\tau_k \equiv \tau$, $\sigma_{k+1} \equiv \sigma$, as well as $\tilde\sigma_{k+1}= \inv\kappa\Theta_k\sigma$ for some constant $\tau,\sigma>0$ satisfying for the strong convexity factors $\gamma_k,\rho_k, \tilde\rho_{k+1}$ and the Lipschitz-like factors $\Theta_k$, $\Lambda_k$ from \cref{ass:pd:main}, for some $\kappa \in (0, 1)$ the inequalities
    \[
        \tilde\rho_{k+1}
        \ge
        \frac{1-\kappa\inv\Theta_k}{2\sigma},
        \quad
        1+\gamma_k\tau
        \ge
        \frac{\tau\sigma\norm{K_k}^2}{(1-\kappa)(1+\rho_k\sigma)}
        +
        (1+\rho_k\sigma)\Lambda_k,
        \quad\text{and}\quad
        1 \ge \tau\sigma\norm{K_k}^2.
    \]
     Then \eqref{eq:pd:stepconds1} holds with $\eta_k = \tauTest_k\tau$, $\tauTest_{k+1} = \tauTest_k(1+\rho_k\sigma)$, and $\sigmaTest_k = \tfrac{\tau}{\sigma}\tauTest_k$.
     In this case \cref{thm:pd:main} shows for some initialisation-dependent constant $C_0$ that
    \begin{multline*}
        [F_{1:N}+\gapmod G_{1:N} \circ K_{1:N}](x^{1:N}) - \inf_{\optx^{1:N}\in \PpredictConstr_{1:N}} [F_{1:N} + G_{1:N} \circ K_{1:N}](\optx^{1:N})
        \\
        \le
        C_0
        + \sum_{k=0}^{N-1}\tauTest_k\left(
            \epsilon_{k+1}(1+\gamma_k\tau)
            +\frac{\kappa(1+\sigma\rho_k)\tau}{2\Theta_k\sigma} \tilde\epsilon_{k+1}
        \right).
    \end{multline*}
    Suppose for simplicity that $\sup_k \rho_k \le \MAX \rho$, $\sup_k \gamma_k \le \MAX \gamma$, and $\inf_k \Theta_k \ge \MIN\Theta$ for some $\MAX\rho,\MAX\gamma,\MIN\Theta>0$ (such as when $\rho_k$, $\gamma_k$, and $\Theta_k$ are constant in $k$).
    Then, minding the sum expression \eqref{eq:pd:multi-fv}, where now $\eta_k=\tau\tauTest_k$, this gives for a constant $C>0$ the result
    \begin{multline*}
        \frac{[F_{1:N}+\gapmod G_{1:N} \circ K_{1:N}](x^{1:N})}{\tau\sum_{k=0}^{N-1} \tauTest_k}
        - \inf_{\optx^{1:N}\in \PpredictConstr_{1:N}} \frac{\sum_{k=0}^{N-1} \tauTest_k[F_{k+1} + G_{k+1} \circ K_{k+1}](\nexxt{\optx})}{\sum_{k=0}^{N-1} \tauTest_k}
        \\
        \le
        C_0
        + C \frac{\sum_{k=0}^{N-1}\tauTest_k\left(
            \epsilon_{k+1}+\tilde\epsilon_{k+1}
        \right)}{\sum_{k=0}^{N-1} \tauTest_k}.
    \end{multline*}
    Exact interpretation requires being able to calculate $\gapmod G_{1:N}$, however, as in \cref{ex:pd:stepconds:const}, we can roughly interpret two cases:
    \begin{enumerate}[label=(\alph*)]
        \item\label{item:pd:stepconds:constexp:vanish}
        If $\lim_{N \to \infty} \sum_{k=0}^{N-1}\tauTest_k\left(\epsilon_{k+1}+\tilde\epsilon_{k+1}\right)/\sum_{k=0}^{N-1} \tauTest_k=0$, the left hand side converges below zero as $N \to \infty$. Roughly, subject to how well we can measure with $\gapmod G_{1:N}$ in place of $G_{1:N}$, this says that asymptotically $x^{1:N}$ are at least as good solutions of the \emph{weighted-}averaged problem $\inf_{x^{1:N}} \frac{1}{\sum_{k=0}^{N-1} \tauTest_k}\sum_{k=0}^{N-1} \tauTest_k[F_{k+1} + G_{k+1} \circ K_{k+1}](\nextx)$ as the best constrained $\optx^{1:N} \in \PpredictConstr_{1:N}$.
        \item\label{item:pd:stepconds:constexp:bound}
        If  $\sup_N \sum_{k=0}^{N-1}\tauTest_k\left(\epsilon_{k+1}+\tilde\epsilon_{k+1}\right)/\sum_{k=0}^{N-1} \tauTest_k \le \delta$ for a constant $\delta$, then, subject to how well we can measure with $\gapmod G_{1:N}$ in place of $G_{1:N}$, this says that asymptotically $x^{1:N}$ stay “within \emph{weighted-}average noise level” $\delta$ of the best $\optx^{1:N} \in \PpredictConstr_{1:N}$.
    \end{enumerate}
    Since $\tauTest_{k+1} = \tauTest_k(1+\rho_k\sigma)$ is increasing, later iterates are weighted more.
    If $\inf_k \rho_k>0$, then $\tauTest_k$ grows exponentially, so the later iterates have exponentially more importance. Thus we can make worse measurements of the early data frames without significantly affecting the quality of the later iterates. If $\epsilon_{k+1}$ and $\tilde\epsilon_{k+1}$ are noise levels of the measurements $A_k$ and $B_k$ of some true temporal coupling operators $\opt A_k$ and $\opt B_k$, the noise levels have to converge to zero for \cref{item:pd:stepconds:constexp:vanish} or stay bounded for \cref{item:pd:stepconds:constexp:bound}.
\end{example}

\section{Optical flow}
\label{sec:flow}

We now apply the previous sections to optical flow.
For numerical accuracy, we use the more fundamental displacement field model instead of the linearised PDE model (transport equation).
For simplicity, and to keep the static problems convex, we concentrate on constant-in-space (but not time) displacement fields. This makes our work applicable to computational image stabilisation (shake reduction) in still or video cameras, compare \cite{tico2009digital,zhou2016image}, based on rapid successions of very noisy images.
We start in \cref{sec:flow:known} with a known displacement field---as could be estimated using acceleration sensors on cameras. Afterwards in \cref{sec:flow:unknown} we include the estimation of the displacement field into our model.

\subsection{Known displacement field}
\label{sec:flow:known}

Denoting by $\delta>0$ the noise level, we start by assuming to be given in each frame, i.e., on each iteration, a noisy measurement $\this b_\delta \in X$ of a true image $\this{\opt b} \in X$ and a noisy measurement $\this v_\delta \in V$ of a true displacement field $\this{\opt v} \in V$. We assume the measured displacement fields $\this v_\delta$ bijective. The finite-dimensional subspaces $X \subset L^2(\Omega)$, $Y \subset L^2(\Omega; \R^2)$, and $V \subset L^2(\Omega; \Omega) \isect C^2(\Omega; \Omega)$ on a domain $\Omega \subset \R^2$ we equip with the $L^2$-norm. We write \eqref{eq:intro:rof} in min-max form with
\begin{subequations}%
\label{eq:flow:pd-known}
\begin{equation}
    \label{eq:flow:pd-known:fns}
    F_k^\delta(x) \defeq \frac{1}{2}\norm{\this b_\delta-x}_{X}^2,
    \quad
    (G_k^\alpha)^*(y) \defeq \delta_{\alpha\B}(y),
    \quad\text{and}\quad
    K_k=D,
\end{equation}
for $\B$ the product of pointwise unit balls and $D: X \to L^2(\Omega)$ a discretised differential operator.
For the primal and dual predictors we take
\begin{equation}
    \label{eq:flow:pd-known-predict}
    A_k^\delta(x) \defeq x \circ \this v_\delta
    \quad\text{and}\quad
    B_k^\delta(y) \defeq y \circ \this v_\delta,
\end{equation}
\end{subequations}
In the dual predictor of the POPD, we take $\tilde G_k^*=(G_k^\alpha)^*+\tfrac{\tilde\rho_k}{2}\norm{\freevar}_{L^2(\Omega)}^2$ following \cref{ex:pd:standard-tilde-g}.
Thus $\tilde G_k^*$ is the Fenchel conjugate of the Huber/Moreau--Yosida-regularised $1$-norm.

\subsubsection*{Regarding the regret and regularisation theory}

Let the true displacement fields $\opt v^k \in H^1(\R^2; \R^2)$, ($k \in \N$), and let $\PDpredictConstr_0 \subset X \times Y$ be bounded. To satisfy \cref{ass:pd:main}\,\cref{item:pd:main-assumption:predictconstr}, we take for some $M>0$,
\begin{subequations}
\label{eq:flow:preditconstr}
\begin{equation}
    \PDpredictConstr \defeq
    \left\{
            \optu^{0:\infty}
        \,\middle|\,
            \begin{array}{l}
            \optu^0 \in \PDpredictConstr_0,\,
            \nexxt\optx = \this\optx \circ \this{\opt v},\,
            \opt x^k \in H^1(\R^2),\,
            \opt y^k \in H^1(\R^2; \R^2),\,
            \norm{\grad \opty^k}_{2,\infty}^2 \le M,
            \\
            \norm{\grad \optx^k}_{2,\infty}^2 \le M,\,
            \nexxt\opty = \prox_{\tilde\sigma_{k+1} \tilde G_{k+1}^*}(\this{\opt y} \circ \this{\opt v} + \tilde\sigma_{k+1} K_{k+1} \nexxt\optx),\,
            \forall k \ge 0
            \end{array}
    \right\}
\end{equation}
as the comparison set.
With a slight abuse of notation we also write $\PDpredictConstr$ for the corresponding set with the domain of each $\this\optu$ restricted to $\Omega$.
We assume that the ground-truth images
\begin{equation}
    \label{eq:flow:optb}
    \opt b^{0:\infty} \in  \PpredictConstr \defeq \{\optx^{0:\infty} \mid \optu^{0:\infty} \in \PDpredictConstr\}.
\end{equation}
\end{subequations}
Because the iterates $\thisy$ are in a finite-dimensional subspace, bounding $\norm{\grad \opty^k}_{2,\infty}^2$ is no difficulty.

To satisfy \eqref{eq:pd:primal-prediction-bound}, we need to find factors $\Lambda_k^\delta \ge 0$ and penalties $\epsilon_{k+1}^\delta \in \R$ such that
\begin{equation}
    \label{eq:flow:primal-prediction-bound}
    \frac{1}{2} \norm{\thisx_\delta \circ \thisv_\delta - \this\optx \circ \this{\opt v}}_{X}^2 \le \frac{\Lambda_k^\delta}{2} \norm{\thisx_\delta -\this\optx}_{X}^2 + \epsilon_{k+1}^\delta
    \quad
    (\this\optx \in \PpredictConstr_k).
\end{equation}
The satisfaction of \eqref{eq:pd:dual-prediction-bound} is handled analogously.
If we had no displacement field measurement error, i.e., $v^k_\delta=\opt v^k$, we could by the area formula take $\Lambda_k^\delta=\max_{\xi \in \Omega} \abs{\det{\grad \inv{(v^k_\delta)}(\xi)}}$ and $\epsilon_{k+1}^\delta=0$.
Otherwise we need the more elaborate estimate of the next lemma.

\begin{lemma}
    \label{lemma:flow:known-flow-bound}
    Let $\opt v \in H^1(\R^2; \R^2)$ and $\optx \in H^1(\Omega)$ with $\norm{\grad \optx}_{2,\infty}^2 \le M$ for some $M>0$.
    Let $\InvDisplacements \subset H^1(\Omega; \Omega)$ be a set of bijective displacement fields satisfying
    \begin{equation}
        \label{eq:flow:lambda-def}
        \Lambda_\InvDisplacements \defeq \sup_{v \in \InvDisplacements,\, \xi \in \Omega} \abs{\det{\grad \inv v(\xi)}} < \infty.
    \end{equation}
    Then for any $x \in L^2(\Omega)$, $v \in \InvDisplacements$, and $\Lambda > \Lambda_\InvDisplacements$,
    \[
        \frac{1}{2} \norm{x \circ v-\optx \circ \opt v}_{L^2(\Omega)}^2
        \le \frac{\Lambda}{2} \norm{x-\optx}_{L^2(\Omega)}^2 + \frac{\Lambda_\InvDisplacements(4\Lambda-3\Lambda_\InvDisplacements)}{8(\Lambda-\Lambda_\InvDisplacements)} M \norm{v-\opt v}_{L^2(\Omega; \R^2)}^2.
    \]
\end{lemma}

\begin{proof}
    By the area formula and Young's inequality, for any $t>0$,
    \begin{multline*}
        \int_\Omega \abs{x(v)-\opt x(\opt v)}^2 \d \xi
        \le
        \int_\Omega \left(1+\tfrac{t}{2}\right)\abs{x(v(\xi))-\opt x(v(\xi))}^2
            + \left(1+\tfrac{1}{2t}\right)\abs{\opt x(v(\xi))-\opt x(\opt v(\xi))}^2 \d \xi
        \\
        =
        \left(1+\tfrac{t}{2}\right)\int_\Omega \abs{x(\xi)-\opt x(\xi)}^2 \abs{\det{\grad \inv v(\xi)}} \d \xi
        +
        \left(1+\tfrac{1}{2t}\right) \int_\Omega \abs{\opt x(v)-\opt x(\opt v)}^2 \d \xi.
    \end{multline*}
    Using \eqref{eq:flow:lambda-def} and that $\optx$ is $\sqrt{M}$-Lipschitz, it follows
    \[
        \norm{x \circ v - \optx \circ \opt v}_{L^2(\Omega)}^2
        \le
        \left(1+\tfrac{t}{2}\right)\Lambda_\InvDisplacements \norm{x-\optx}_{L^2(\Omega)}^2
        +\left(1+\tfrac{1}{2t}\right) M \norm{v-\opt v}_{L^2(\Omega; \R^2)}^2.
    \]
    Taking $t=2(\Lambda/\Lambda_\InvDisplacements-1)$ yields the claim.
\end{proof}

We need the primal iterates to stay bounded. For this we use the next lemma:

\begin{lemma}
    \label{lemma:flow:bound-one}
    Compute $\thisx_\delta$ and $\this\upsilon_\delta$ by \cref{alg:pd:alg} for \eqref{eq:flow:pd-known:fns} with fixed $\delta>0$ and $\tau_k \equiv \tau>0$.
    Suppose $\tau \le \frac{(2-\Lambda)C-\epsilon}{\alpha^2\norm{D}^2}$ and $\norm{\xi_\delta^k-b_\delta^k}^2 \le C \Lambda + \epsilon$ for some $C,\Lambda,\epsilon>0$. Then $\norm{x_\delta^k-b_\delta^k}^2 \le C$.
\end{lemma}

\begin{proof}
    We drop the indexing by $\delta$ as it is fixed.
    The dual prediction
    of \cref{alg:pd:alg} guarantees $\norm{\upsilon^k}_{2,\infty} \le \alpha$.
    The primal step
    is
    \begin{equation}
        \label{eq:flow:primalbound:first}
        x^k \defeq \argmin_x \norm{x-\xi^k-\tau D \upsilon^k}^2 + \tau\norm{x-b^k}^2.
    \end{equation}
    The optimality conditions are $0 = x^k - \xi^k +\tau D \upsilon^k + \tau(x^k - b^k)$. Thus $\tau\norm{x^k-b^k}=\norm{x^k - \xi^k +\tau D \upsilon^k}$. By \eqref{eq:flow:primalbound:first}, comparing to $x=\xi^k$, we get
    \[
        2\norm{x^k-b^k}^2
        \le
        \tau\norm{D \upsilon^k}^2 + \norm{\xi^k-b^k}^2
        \le \tau \alpha^2\norm{D}^2 + C \Lambda + \epsilon.
    \]
    Thus $\norm{x^k-b^k}^2 \le C$ when $\tau$ is as stated.
\end{proof}

We may now prove convergence to the true data as the displacement field measurement error $\epsilon \downto 0$ along with the noise in the data $b^k_\delta$.

\begin{theorem}
    \label{thm:flow:asym}
    For all $k \in \N$, $\delta>0$, and some $\alpha=\alpha(\delta) \downto 0$ as $\delta \downto 0$, assume the setup of \eqref{eq:flow:pd-known} and \eqref{eq:flow:preditconstr} with $\this v_\delta \in \InvDisplacements$ for a set $\InvDisplacements \subset V$ of bijective displacement fields such that $\Lambda_\InvDisplacements<2$.
    With $\opt b^{0:\infty} \in \PpredictConstr$, assume:
    \begin{enumerate}[label=(\Roman*),noitemsep]
        \item\label{item:flow:known:noise}
        $\sup_{k \in \N} \norm{\this b_\delta - \this{\opt b}}_{L^2(\Omega)} \to 0$ and $\sup_{k \in \N} \norm{\this v_\delta-\this{\opt v}}_{L^2(\Omega; \R^2)} \to 0$ as $\delta \downto 0$.

        \item\label{item:flow:known:step}
        For some $\Lambda_k,\Theta_k\equiv\Lambda>\Lambda_\InvDisplacements$, the step length parameters are as in \cref{ex:pd:stepconds:const}, independent of $\delta$ and $k$.
    \end{enumerate}
    For an initial $u^0=u^0_\delta$, for all $\delta>0$, generate $u^{1:\infty}_\delta$ by \cref{alg:pd:alg}. Then there exist $\bar N(\delta) \in \N$ such that:
    \begin{enumerate}[label=(\alph*),noitemsep]
        \item\label{item:flow:basic-result}
        $\lim_{\delta \downto 0} \sup_{N \ge \bar N(\delta)} \frac{1}{N}\sum_{k=1}^{N} \norm{\thisx_\delta-\this{\opt b}}_{L^2(\Omega)}^2 = 0$, and
        \item\label{item:flow:estim-result}
        provided $\tau\sigma\norm{D}^2 < 1$, moreover, $\lim_{\delta \downto 0} \sup_{N \ge \bar N(\delta)} \frac{1}{2N}\sum_{k=0}^{N-1}\norm{\nextx_\delta-\thisx_\delta \circ \this v_\delta}_{L^2(\Omega)}^2 = 0$.
    \end{enumerate}
\end{theorem}

\begin{proof}
    We first show the boundedness of $\{\thisx_\delta\}_{k \in \N, \delta \in (0, \bar\delta)}$ for some $\bar\delta>0$.
    By \ref{item:flow:known:noise},
    $\sup_{k \in \N} \norm{\this b_\delta - \this{\opt b}}_X =: \delta_b \to 0$
    and $\sup_{k \in \N} \norm{\this v_\delta-\this{\opt v}}_{L^2(\Omega; \R^2)} =: \delta_v \to 0$ as $\delta \downto 0$.
    We have $\nexxt{\opt b} = \this{\opt b} \circ \this{\opt v}$ and $\nexxt\xi_\delta = \thisx_\delta \circ \this v_\delta$. Using Young's inequality twice for any $\beta>0$ and \cref{lemma:flow:known-flow-bound} for any $\Lambda'>\Lambda_\InvDisplacements$,
    \[
        \begin{aligned}
        \norm{\nexxt\xi_\delta-\nexxt b_\delta}_{L^2(\Omega)}^2
        &
        \le (1+\beta)\norm{\thisx_\delta \circ \this v_\delta - \this{\opt b} \circ \this{\opt v}}_{L^2(\Omega)}^2  +(1+\inv\beta)\norm{\nexxt b_\delta -\this{\opt b} \circ \this{\opt v}}_{L^2(\Omega)}^2
        \\
        &
        \le \Lambda'(1+\beta)\norm{\thisx_\delta - \this{\opt b}}_{L^2(\Omega)}^2 + (1+\inv\beta)\delta_b^2 + (1+\beta)\tfrac{\Lambda_\InvDisplacements(4\Lambda'-3\Lambda_\InvDisplacements)}{8(\Lambda'-\Lambda_\InvDisplacements)} M  \delta_v^2
        \\
        \MoveEqLeft[5]
        \le \Lambda'(1+\beta)^2\norm{\thisx_\delta - \this{b}_\delta}_{L^2(\Omega)}^2 + (1+\inv\beta)\left(\Lambda'(1+\beta) +1\right)\delta_b^2 + (1+\beta)\tfrac{\Lambda_\InvDisplacements(4\Lambda'-3\Lambda_\InvDisplacements)}{8(\Lambda'-\Lambda_\InvDisplacements)} M  \delta_v^2.
        \end{aligned}
    \]
    Taking $\beta>0$, $\Lambda'>\Lambda_\InvDisplacements$ small enough, we obtain for any $\Lambda \in (\Lambda_\InvDisplacements, 2)$ that $\norm{\nexxt\xi_\delta-\nexxt b_\delta}^2 \le \Lambda \norm{\thisx_\delta - \this{b}_\delta}^2 + \epsilon_\delta$ for some $\epsilon_\delta \downto 0$ as $\delta \downto 0$.
    By \cref{lemma:flow:bound-one}, now $\sup_k \norm{x_\delta^k - b_\delta^k}^2 \le C$ for any $C \ge \norm{x^0 - b_\delta^0}^2$ with $\tau \le \frac{(2-\Lambda)C-\epsilon_\delta}{\alpha^2\norm{D}^2}$. This holds for $C$ large and $\delta \in (0, \bar\delta)$ for small $\bar\delta>0$.
    Thus $\sup_{k \in \N, \delta \in (0, \bar\delta)} \norm{\thisx_\delta} < \infty$.

    Fix now $\delta \in (0, \bar\delta)$ and $N \ge 1$. By \cref{lemma:flow:known-flow-bound} and \cref{item:flow:known:step}, the prediction bounds \cref{eq:pd:prediction-bound} hold for all $k \in \N$ with $\Lambda_k=\Theta_k \equiv \Lambda$ and
    \begin{equation}
        \label{eq:flow:epsilon}
        \epsilon_{k+1} = \tilde\epsilon_{k+1} = \epsilon_{k+1}^\delta
        \defeq \frac{\Lambda_\InvDisplacements(4\Lambda_k-3\Lambda_\InvDisplacements)}{8(\Lambda_k-\Lambda_\InvDisplacements)} M \norm{\this v_\delta-\this{\opt v}}_{L^2(\Omega; \R^2)}^2
        \le \frac{\Lambda_\InvDisplacements(4\Lambda-3\Lambda_\InvDisplacements)}{8(\Lambda-\Lambda_\InvDisplacements)} M \delta_v.
    \end{equation}
    The rest of \cref{ass:pd:main} holds by the construction in \eqref{eq:flow:pd-known} and \eqref{eq:flow:preditconstr} while \cref{eq:pd:stepconds1} holds by \cref{item:flow:known:step} and \cref{ex:pd:stepconds:const}.
    By \eqref{eq:linpred:pd:zm} and \cref{ex:pd:stepconds:const}, also $\Test_k\Precond_k \ge \begin{psmallmatrix} 1-\tau\sigma\norm{D}^2 & 0 \\ 0 & 0 \end{psmallmatrix} \ge 0$.
    Therefore, by \cref{thm:pd:main}, for some constant $C>0$ (dependent on the initialisation, $\sup_{\delta \in (0, \bar\delta)} \delta_v$, $\Theta_k \equiv \Lambda$, and $\rho_k \equiv 0$ as well as $\tauTest_k \equiv 1$ and $\sigmaTest_k \equiv \tfrac{\tau}{\sigma}$ as in \cref{ex:pd:stepconds:const}), we have with the notation $F_{1:N}$ etc.~from \cref{thm:pd:main} that
    \begin{multline*}
        \frac{1}{N}[F_{1:N}^\delta+\gapmod G_{1:N}^{\alpha(\delta)} \circ K_{1:N}](x_\delta^{1:N})
        -\frac{1}{N}\inf_{\optx^{1:N}\in \PpredictConstr_{1:N}} [F_{1:N}^\delta+G_{1:N}^{\alpha(\delta)} \circ K_{1:N}](\optx^{1:N})
        \\
        +\sum_{k=0}^{N-1} \frac{1-\tau\sigma\norm{D}}{2N} \norm{\nextx_\delta-\thisx_\delta \circ \this v_\delta}^2
        \le
        \frac{C}{N}.
    \end{multline*}
    By \cref{lemma:gaps:indicator}, the defining \eqref{eq:flow:pd-known}, and the just proved boundedness of the iterates,
    \[
        \gapmod G_{1:N}^{\alpha(\delta)}(K_{1:N}x_\delta^{1:N}) \ge - G_{1:N}^{\alpha(\delta)}(-K_{1:N}x_\delta^{1:N})
        = -\sum_{k=1}^N \alpha(\delta)\norm{Dx^k_\delta}_B \ge - \alpha(\delta) N C'
    \]
    for some constant $C'>0$.
    Since $\PpredictConstr$ is bounded (by the boundedness of $\PDpredictConstr_0$ and finite-dimensionality), also $[G_{1:N}^{\alpha(\delta)} \circ K_{1:N}](\optx^{1:N}) \le \alpha \bar C'$ for some $\bar C'>0$.
    Hence, for some $C''>0$ we get for all $\optx^{1:N}\in \PpredictConstr_{1:N}$ that
    \[
        \sum_{k=1}^N \left(\tau\frac{F_k^\delta(\thisx_\delta)-F_k^\delta(\optx^k)}{N}+ \frac{1-\tau\sigma\norm{D}}{2N} \norm{\nextx_\delta-\thisx_\delta \circ \this v_\delta}^2\right)
        \le
        \alpha(\delta) C'' + \frac{C}{N}.
    \]
    Due to \eqref{eq:flow:optb}, this says for all $\delta \in (0, \bar\delta)$ and $N \ge 1$ that
    \[
        \sum_{k=1}^N\left( \frac{\tau}{2N}\norm{\this{b_\delta}-\thisx}^2
        +\frac{1-\tau\sigma\norm{D}}{2N} \norm{\nextx_\delta-\thisx_\delta \circ \this v_\delta}^2
        \right)
        \le
        \sum_{k=1}^N \frac{\tau}{2N}\norm{\this{b_\delta}-\this{\opt b}}^2
        + \alpha(\delta) C'' + \frac{C}{N}.
    \]
    Since $\alpha(\delta) \downto 0$ and \cref{item:flow:known:noise} guarantees $\sup_{k \in \N} \norm{\this b_\delta - \this{\opt b}}_{L^2(\Omega)} \to 0$ as $\delta \downto 0$, the right hand side can be made smaller than $\delta$ by taking $N \ge N(\delta)$ large enough.
    The claim \cref{item:flow:estim-result} immediately follows while \cref{item:flow:basic-result} follows after further referral to \cref{item:flow:known:noise}.
\end{proof}

\subsubsection*{Numerical setup}

\begin{figure}%
    \centering%
    \begin{subfigure}{0.33\textwidth}%
        \includegraphics[width=\textwidth]{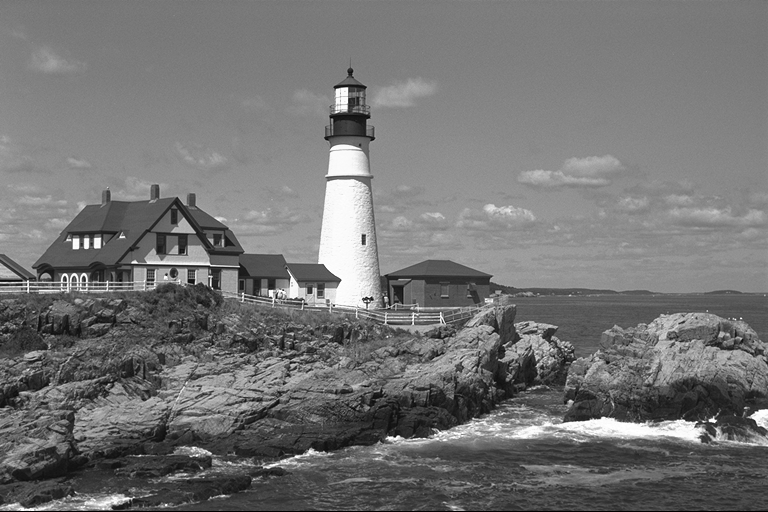}%
        \caption{Original}
        \label{fig:flow:lighthouse:tv:orig}
    \end{subfigure}\,%
    \begin{subfigure}{0.33\textwidth}%
        \includegraphics[width=\textwidth]{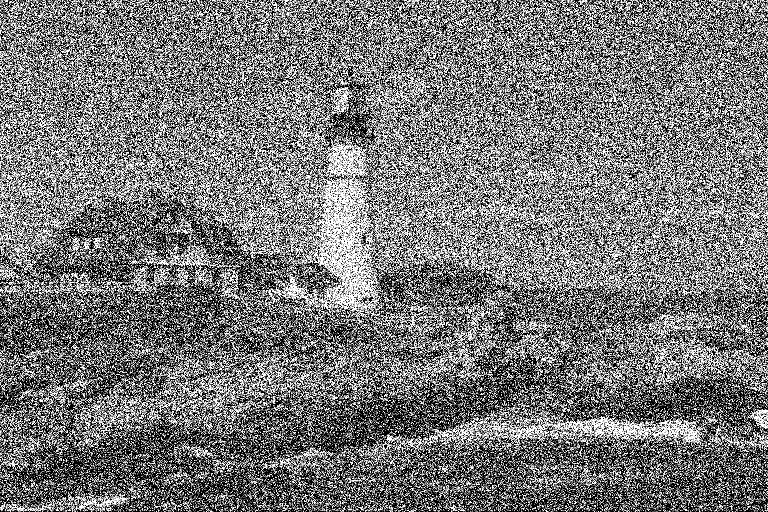}%
        \caption{Data: noise level $0.5$}
    \end{subfigure}\,%
    \begin{subfigure}{0.33\textwidth}%
        \includegraphics[width=\textwidth]{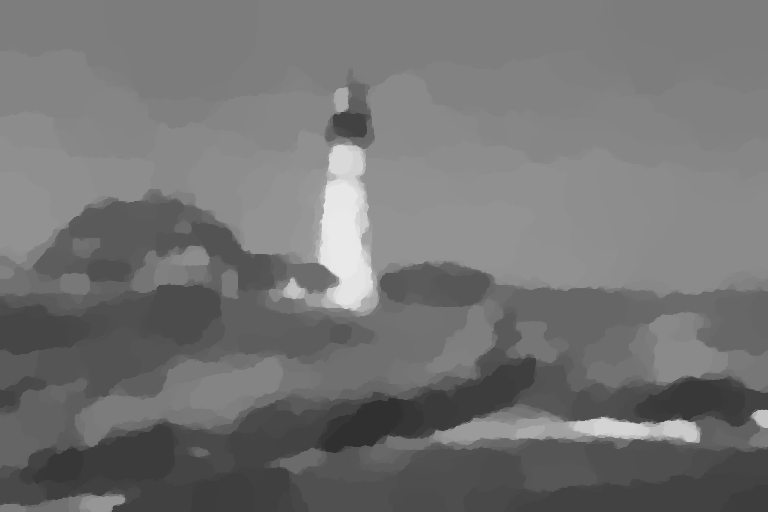}%
        \caption{Recon.: noise level $0.5$, $\alpha=1$}
        \label{fig:flow:lighthouse:tv:known}
    \end{subfigure}%
    \caption{Test image, added noise, and stationary reconstruction for comparison.}
    \label{fig:flow:lighthouse}
\end{figure}

We perform our experiments on a simple square image as well as the lighthouse image from the free Kodak image suite \cite{franzenkodak}; this is in \cref{fig:flow:lighthouse} along with a noisy version and comparison single-frame total variation reconstruction. The original size is 768$\times$512 pixels. For our experiments, we pick a $300 \times 200$ subimage moving according to Brownian motion of standard deviation $2$. Thus the displacement fields $\this{\opt v}(\xi)=\xi - \this{\opt u}$ with $\this{\opt u} \in \R^2$ are constant in space.
To the subimage we add 50\% Gaussian noise (standard deviation 0.5 with original intensities in $[0, 1]$).
To construct the measured displacements available to the algorithm we add 5\% Gaussian noise (standard deviation $0.05\norm{\this{\opt u}}$) to the true displacements.\footnote{
    Then \eqref{eq:flow:lambda-def} gives $\Lambda_\InvDisplacements=1$.
    Constant true displacements are allowed by \cref{lemma:flow:known-flow-bound}, but constant measurements not. If $\norm{x-\optx}_{L^2(\Omega+ B(0, \norm{u}))}^2 \le C \norm{x-\optx}_{L^2(\Omega)}^2$ then \cref{lemma:flow:known-flow-bound,thm:flow:asym} extend to $\Lambda > C\Lambda_\InvDisplacements$.
    In practise, to compute $x \circ v$, we extrapolate $x$ outside $\Omega$ such that Neumann boundary conditions are satisfied.
}

We take the regularisation parameter $\alpha=1$. The corresponding full-image total variation reconstruction is in \cref{fig:flow:lighthouse:tv:known}. To parametrise the POPD (\cref{alg:pd:alg}) we
\begin{itemize}[label={--},nosep]
    \item Fix the primal step length parameter $\tau=0.01$ as well as $\Lambda=\Theta=1$ and $\kappa=0.9$.
    \item Take the primal strong convexity factor $\gamma=1$ and generally the dual factor $\rho=0$.
    \item Take $\tilde\sigma$, maximal $\sigma$, and minimal $\tilde \rho_{k+1} \equiv \tilde \rho$ according to \cref{ex:pd:stepconds:const}. Here we estimate $\norm{K_k} \le \sqrt{8}$ for forward-differences discretisation of $K_k=D$ with cell width $h=1$ \cite{chambolle2004algorithm}.
\end{itemize}
Although $G_{k+1}^*$ is not strongly convex, we also experiment taking a “phantom” $\rho=100$. This can in principle be justified via \emph{local} strong convexity or \term{strong metric subregularity} at a solution. We briefly indicate how this works in \cref{app:localstrong}. The effect in practise is to increase the dual step length parameter $\sigma$.
We always take zero as the initial iterate (primal and dual).

We implemented our algorithms in Julia 1.3 \cite{bezanson2017julia}, and performed our experiments on a mid-2014 MacBook Pro with 16GB RAM and two CPU cores.
Our implementation uses a maximum of four computational threads (two cores with hyperthreading) in those parts of the code where this appears advantageous. The data generation runs in its own thread.
The implementation is available on Zenodo \cite{tuomov-predict-code}.

\subsubsection*{Numerical results}

We display the reconstructions in \cref{fig:flow:square:known,fig:flow:lighthouse:known,fig:flow:lighthouse:known-largesigma} and the performance (function value, PSNR, and SSIM) in \cref{fig:flow:values:known,fig:flow:psnr:known,fig:flow:ssim:known}. The reconstructions are for the frames/iterations 30, 50, 100, 300, 500, 1000, and 3000 whereas the performance plots display all 10000 iterations at a resolution of 100 iterations after the first 100 iterations. The right-most column of the reconstruction figures displays the true cumulative displacement field up to the corresponding data frame (indicated in the bottom-left corner). The darker line is sampled at the same resolution as the  performance plots whereas the lighter line is sampled at every iteration. Regarding real-time computability, averaged over the 10000 iterations, every iteration takes $\sim$6.5ms, which is to say the POPD can process 154 frames per second.

The performance plots show convergence of the function value to a stable value, not necessarily a minimum, within 100 iterations. Likewise the SSIM and PSNR reach a relatively stable and acceptable value by 100 iterations.
Visually, we have decent tracking of movement, but we need the large $\rho$-value to get a noticeable cartoon-like “total variation effect”. In the last frame of \cref{fig:flow:square:known} we can see the effect of the algorithm not being able to track a sudden large displacement fast enough, hence producing some motion blur. The 100 iterations, that were needed to reach a stable function value, SSIM, or PSNR, appear to be mainly needed to reach the correct contrast level: recall that we initialise with zero. We tested initialising the primal variable with the noisy data: the algorithm then  needed a similar number of iterations to reduce the noise. A smarter initialisation might help reduce the 100-iteration ”initialisation window”.

For comparison, we have included POFB reconstruction (\cref{alg:fb:alg}) in \cref{fig:flow:lighthouse:pofb}. We use the step length parameter $\tau=0.01$ for the POFB itself.
We take 10 iterations of FISTA \cite{beck2009fista} with step length parameter $\tilde\tau=1/\norm{K}^2$ to approximately solve the proximal step. By the performance measures the results are comparable to the POPD. Visually they are similar to the high-$\rho$ POPD. The algorithm is, however, quite a bit slower: $\sim$21.2ms/frame or 47 frames per second. Solving the proximal step accurately would further slow it down.

\subsection{Unknown displacement field}
\label{sec:flow:unknown}

When the displacement field $v_k$ is completely unknown, we need to estimate it from data.
For some $E_k:V \to \extR$ we do this through
\begin{equation}
    \label{eq:flow:unknown:prob}
    \min_{x \in X,\, v \in V} \frac{1}{2}\norm{b_k-x}_X^2
    + \alpha\norm{Dx} + E_k(v)
\end{equation}
We drop the indexing by the noise level $\delta>0$ as we will not be studying regularisation properties.
Ideally we would take $E_k(v)$ as $\frac{\theta}{2}\norm{b^{k+1}-b^k \circ v}_X^2$, plus regularisation terms. However, the resulting problem would be highly nonconvex. A second idea is to use a Horn--Schunck \cite{horn1981determining} type penalty on linearised optical flow\footnote{%
    To obtain the linearised optical flow model, we start with $b_{k+1}(\xi)=b_{k}(v_k(\xi))$ holding for all $\xi \in \Omega$ and a sufficiently smooth image $b_{k}$.
    By Taylor expansion $b_{k}(v_k(\xi)) \approx b_{k}(\xi) + \iprod{\grad b_{k}(\xi)}{v_k(\xi)-\xi}$. Thus $0 = b_{k+1}(\xi)-b_{k}(v_k(\xi)) \approx b_{k+1}(\xi)-b_{k}(\xi) + \iprod{\grad b_{k}(\xi)}{\xi-v_k(\xi)}$.
}, taking for some parameters $\theta,\lambda_1,\lambda_2>0$,
\begin{equation}
    \label{eq:flow:unknown:first-attempt}
    E_k(v) =
    \frac{\theta}{2}\norm{b_{k+1}-b_k + \pointwiseiprod{\text{Id}-v}{\grad b_k}}_X^2
    + \frac{\lambda_1}{2}\norm{\Id-v}_2^2 + \frac{\lambda_2}{2}\norm{\grad v}_2^2,
\end{equation}
where the pointwise inner product $\pointwiseiprod{a}{b}(\xi) \defeq \iprod{a(\xi)}{b(\xi)}$. We regularise the displacement field $v$ to both be close to identity (no displacement) and to be smooth in space.\footnote{%
    Indeed, in linearised optical flow the displacement field cannot in general be discontinuous. See \cite{tuomov-bd,chen2012image} for approaches designed to avoid this restriction.%
}

The choice \eqref{eq:flow:unknown:first-attempt} is, however, very inaccurate in practise. We therefore, firstly, introduce a time-step parameter $T$ and a convolution kernel $\varrho$ to counteract noise in the data.
Secondly, we average the Horn--Schunck term over a window of $n$ frames. For iteration $k$, the last frame is
\[
    \iota(k) \defeq \max\{1,k+1-(n-1)\}
    \quad\text{and its true length}\quad
    n_k \defeq k+1-(\iota(k)-1).
\]
With $j \in \{\iota(k),\ldots,k+1\}$, we write $v_j^k \in V$ for the displacement of $b^j$ from $b^{\iota(k)-1}$ as estimated on iteration $k$. Then the displacement of $b^{j+1}$ from $b^j$ is $\inv{(v_j^k)} \circ v_{j+1}^k$.
We take $E_k: V^{n_{k+1}} \to \R$,
\begin{equation}
    \label{eq:flow:unknown:efinal}
    \begin{aligned}[t]
    E_k(v_{\iota(k):k+1}^k) &\defeq
    \frac{1}{n_k}\sum_{j=\iota(k)-1}^{k}\Bigl(
    \frac{\theta}{2}\norm{\varrho*(b^{j+1}-b^j)/T + \pointwiseiprod{\Id - \inv{(v_j^k)} \circ v_{j+1}^k}{\grad(\varrho * b^j)}}_X^2
    \\ \MoveEqLeft[-9]
    + \frac{\lambda_1}{2}\norm{\Id - \inv{(v_j^k)} \circ v_{j+1}^k}_2^2 + \frac{\lambda_2}{2}\norm{\grad v_j^k}_2^2
    \Bigr).
    \end{aligned}
\end{equation}
Although not given as a parameter, we use $v_{\iota(k)-1}^k = 0$.

We predict the primal variables using
\[
    A_k(x, v_{\iota(k):k+1}^k)
    \defeq
    \begin{cases}
        (x \circ \inv{(v_{k}^k)} \circ v_{k+1}^k, v_1^k, \ldots, v_{k+1}^k, 0), & k < n, \\
        (x \circ \inv{(v_{k}^k)} \circ v_{k+1}^k, \inv{(v_{\iota(k)}^k)} \circ v_{\iota(k+1)}^k, \ldots, \inv{(v_{\iota(k)}^k)} \circ v_{k+1}^k, 0), & k \ge n,
    \end{cases}
\]
and the dual variables using
\[
    B_k(y) \defeq y \circ \inv{(v_k^k)} \circ v_{k+1}^k
\]
Hence we a) propagate the image $x$ and the dual variable using the estimated displacement of the next frame from the current frame, b) update the displacement estimates to be with respect to the start $\iota(k+1)$ of the new $n$-frame window, and c) predict the displacement between the next two frames to be zero. The latter is consistent with the zero-mean Brownian motion used in our numerical experiments.

We write the problem \eqref{eq:flow:unknown:prob} with $E_k$ given by \eqref{eq:flow:unknown:efinal} in the form \eqref{eq:intro:minmax-sequence} by taking
\[
    F_k(x, v_{\iota(k):k+1}^k) \defeq \frac{1}{2}\norm{b^k-x}_V^2 + E_k(v_{\iota(k):k+1}^k),
    \quad
    K_k(x, v_{\iota(k):k+1}^k) \defeq Dx,
    \quad\text{and}\quad
    G_k^*(y) \defeq \delta_{\alpha\B}(y).
\]
We split $\prox_{\tau F_k}$ into individual updates with respect to $x$ and $v_{\iota(k):k+1}^k$.
If the displacement fields are constant in space, $v_j^k(\xi)=\xi - u_j^k$ with $u_j^k \in \R^2$, the compositions $\inv{(v_{\iota(k)}^k)}  \circ v_j^k \equiv u_{\iota(k)}^k - u_j^k$, and $\prox_{\tau E_k}$ reduces to an easily solvable chain of $2 \times 2$ quadratic optimisation problems.

The Horn--Schunck linearisation of the optical flow only converges to the true optical flow as we increase the temporal resolution. Therefore, an equivalent of the regularisation theory of \cref{thm:flow:asym} for the present model would require increasing the temporal resolution as $\delta \downto 0$ and $N \upto \infty$. As the analysis is somewhat involved, we have decided not to pursue such estimates. It is, however, not difficult to extend the prediction bounds of \cref{lemma:flow:known-flow-bound}.

\subsubsection*{Numerical setup and results}

For our numerical experiments we use generally the same setup as in \cref{sec:flow:known} except we reduce the noise level in the image to 30\% and correspondingly take $\alpha=0.2$. For our new parameters we take $\lambda_1=1$ and $\theta=(300\cdot 200) \cdot 100^3$ with constant-in-space displacement fields, so that $\lambda_2$ is irrelevant in \eqref{eq:flow:unknown:efinal}.
For the displacement estimation we use a window of $n=100$ previous frames.
For the smoothing kernel $\varrho$ in the Horn--Schunck term of \eqref{eq:flow:unknown:efinal} we take a normalised Gaussian of standard deviation $3$ pixels in a window of $11 \times 11$ pixels. We also take the time step parameter $T=0.5$ for the lighthouse and $T=1$ for the square test image.
Our Julia implementation is available on Zenodo \cite{tuomov-predict-code}.

The reconstructions and estimated displacements are in \cref{fig:flow:square:unknown,fig:flow:lighthouse:unknown,fig:flow:lighthouse:unknown-largesigma} and the performance plots (function value, PSNR, SSIM) in \cref{fig:flow:values:unknown,fig:flow:psnr:unknown,fig:flow:ssim:unknown}.
Regarding real-time computability, the POPD requires 20.8ms/iteration, that is, can process 48 frames per second.

The function values take a long time to decrease. The PSNR and SSIM, however, again reach an acceptable and somewhat stable value after 100--200 iterations.
Visually, the results are somewhat more blurred than with the approximately known displacement in \cref{sec:flow:known}, and even with $\rho=100$ the cartoon-like total variation effect remains small. Nevertheless, the reconstructions are visually pleasing and the displacement is estimated to an acceptable accuracy. This did, however, require adapting the time-step parameter $T$ to the test case. Improving the optical flow model to not require such an extraneous parameter is something for future research: we believe that the present results already demonstrate that online optimisation is a worthy approach to dynamic imaging.

\newcounter{figureSave}\setcounter{figureSave}{\value{figure}}
\newcounter{sectionSave}\setcounter{sectionSave}{\value{section}}
\let\thesectionSave\thesection

\section{Conclusion}

With the goal of solving---for now relatively simple---imaging problems “online”, in real-time, we incorporated predictors into the forward-backward and primal-dual proximal splitting methods. For the predictive online forward-backward method (POFB) a reasonable notion of “dynamic regret” stays bounded, and can even converge below zero. Using regularisation theory we, moreover, proved convergence to a ground-truth as the level of corruption in the problem data vanishes. Hence the method forms an appropriate regulariser.

We do not, yet, understand the predictive online primal-dual method (POPD) as well. While we have shown analogous results, including convergence as the data improves, the form of “regret” we were able to employ still requires study and interpretation.
This notwithstanding, our numerical results on optical flow are encouraging. More research is needed to understand the parametrisation and improved predictors needed to make the total variation effect prominent.

\appendix

\section{Local strong convexity}
\label{app:localstrong}

We establish local strong convexity of the indicator function of the ball. This has been shown in \cite{aragon2008characterization} to be equivalent to the \term{strong metric subregularity} of the subdifferential. For related characterisations, see also \cite{tuomov-subreg} and regarding total variation \cite[appendix]{jauhiainen2019gaussnewton}.

\begin{lemma}
    \label{lemma:localstrong:indicator}
    With $F: X \to \R$, $F=\delta_{\closure \B(0, \alpha)}$ on a Hilbert space $X$, suppose $x \in \partial \B(0, \alpha)$ and $0 \ne x^* \in \subdiff F(x)$. Then
    \begin{gather*}
        F(x')-F(x) \ge \iprod{x^*}{x'-x} + \frac{\gamma}{2}\norm{x'-x}^2
        \quad (x' \in U_x)
    \shortintertext{for}
        U_x = \begin{cases}
            X, & 0 \le \gamma\alpha \le \norm{x^*}, \\
            [\closure \B(0,\alpha)]^c \union \closure \B(x, \alpha), & \alpha\gamma > \norm{x^*}.
        \end{cases}
    \end{gather*}
\end{lemma}

\begin{proof}
    Observe that $x^*=\lambda x$ for $\lambda \defeq \norm{x^*}/\alpha$.
    If $x' \not\in \closure \B(0, \alpha)$, there is nothing to prove. So take $x' \in \closure \B(0, \alpha)$. Then we need
    $
        0 \ge \lambda \iprod{x}{x'-x} + \frac{\gamma}{2}\norm{x'-x}^2.
    $
    Since $\norm{x}=\alpha$, this says
    \begin{equation}
        \label{eq:localsc:expanded}
        \left(\lambda-\frac{\gamma}{2}\right)\alpha^2
        \ge
        \frac{\gamma}{2}\norm{x'}^2
        +
        \left(\lambda-\gamma\right)\iprod{x}{x'}.
    \end{equation}

    Suppose $\gamma \le \lambda$, which is the first case of $U_x$. Then  \eqref{eq:localsc:expanded} is seen to hold by application of Young's inequality on the inner product term, followed by $\norm{x'} \le \alpha$.

    If on the other hand, $\gamma > \lambda$, which is the second case of $U_x$, we take $x' \in \closure \B(x, \alpha) \isect \closure \B(0,\alpha)$. This implies
    $
        \iprod{x'}{x} \ge \tfrac{1}{2}\norm{x'}^2.
    $
    Since $\lambda-\gamma<0$, this and $\norm{x'} \le \alpha$ prove \eqref{eq:localsc:expanded}.
\end{proof}

 \providecommand{\eprint}[1]{\href{http://arxiv.org/abs/#1}{arXiv:#1}}
  \providecommand{\eprint}[1]{\href{http://arxiv.org/abs/#1}{arXiv:#1}}
  \providecommand{\noopsort}[1]{}


\setcounter{figure}{\value{figureSave}}
\setcounter{section}{\value{sectionSave}}
\makeatletter
\let\thesection\thesectionSave
\makeatother


\def\exponeknown{square200x300_pdps_known_7600d78920253b23}
\def\exptwoknown{lighthouse200x300_pdps_known_7600d78920253b23}
\def\expthreeknown{lighthouse200x300_pdps_known_2b70dcc3b25c0e99}
\def\exppofbknown{lighthouse200x300_fb_known_f7b3704b57a8c02d}

\def\exponeunknown{square200x300_pdps_unknownmulti_f99729d381579f48}
\def\exptwounknown{lighthouse200x300_pdps_unknownmulti_f99729d381579f48}
\def\expthreeunknown{lighthouse200x300_pdps_unknownmulti_537a9ed48254733e}

\begin{figure}%
    \dostrip{\exponeknown}
    \input{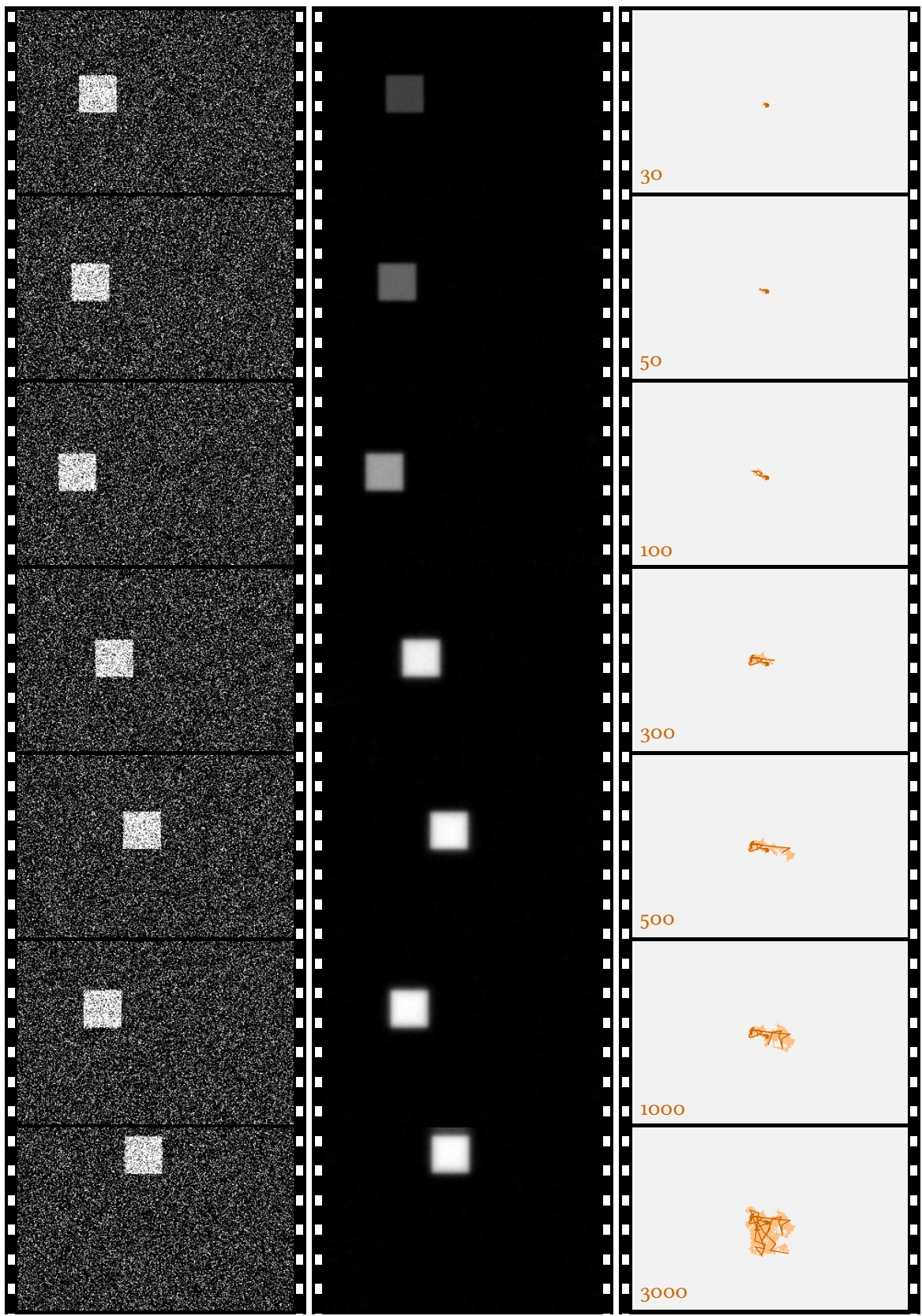}
    \caption{Square, POPD, approximately known displacement, $\rho=\EXPPARAMphantomrho$.}
    \label{fig:flow:square:known}
\end{figure}

\begin{figure}%
    \dostrip{\exptwoknown}
    \input{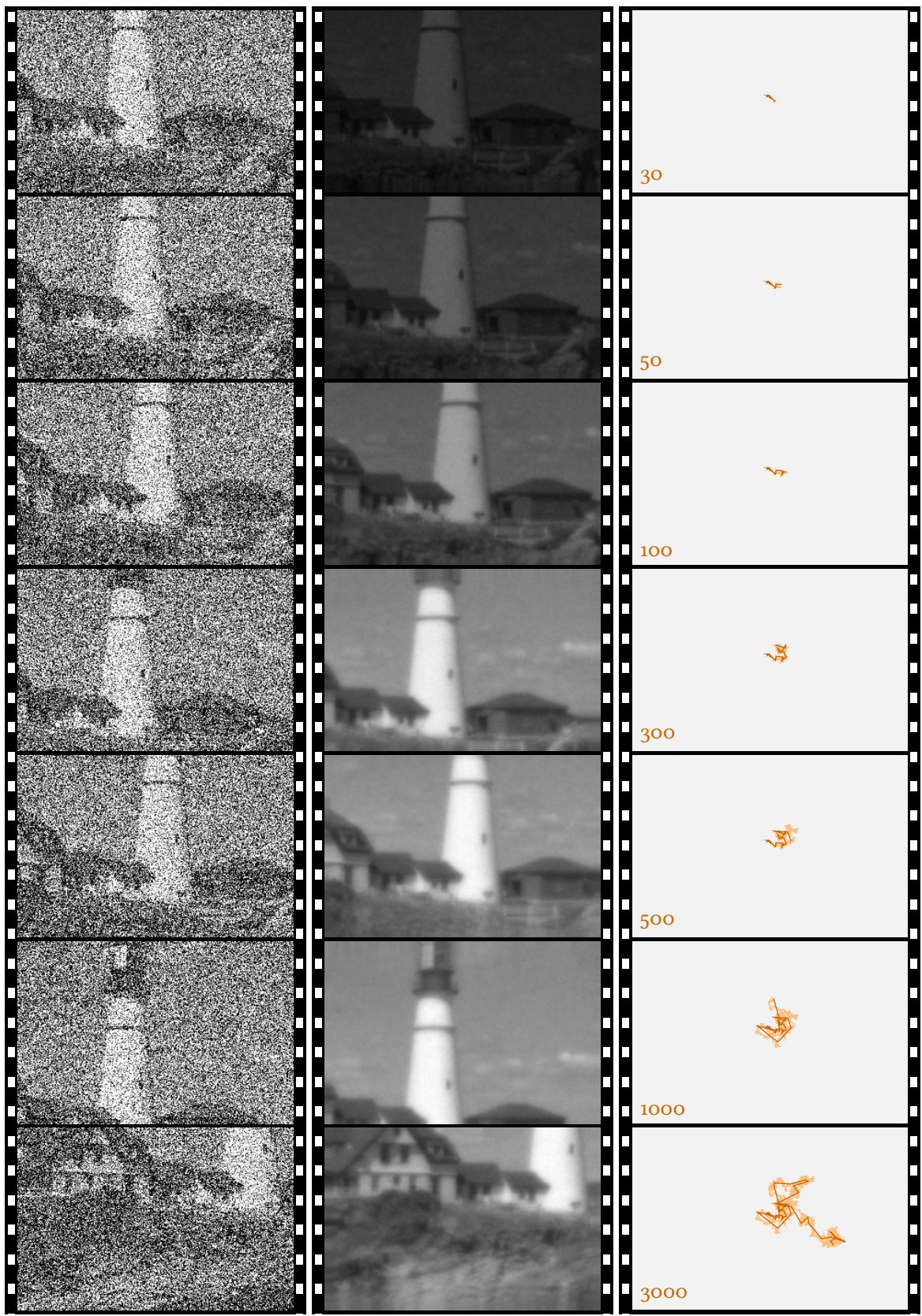}
    \caption{Lighthouse, POPD, approximately known displacement, $\rho=\EXPPARAMphantomrho$.}
    \label{fig:flow:lighthouse:known}
\end{figure}

\begin{figure}%
    \dostrip{\expthreeknown}
    \input{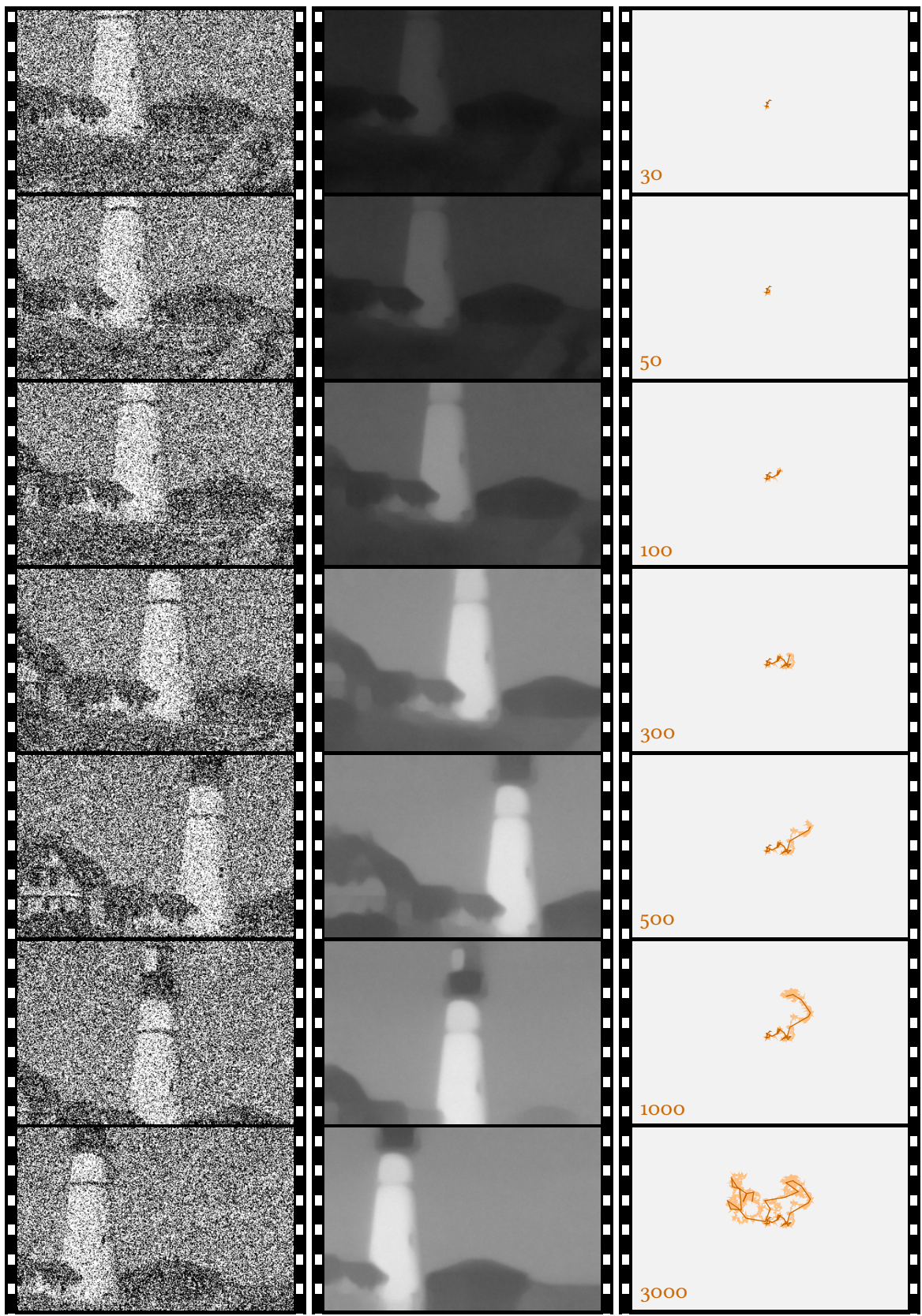}
    \caption{Lighthouse, POPD, approximately known displacement,  $\rho=\EXPPARAMphantomrho$.}
     \label{fig:flow:lighthouse:known-largesigma}
\end{figure}

\begin{figure}%
    \dostrip{\exppofbknown}
    \input{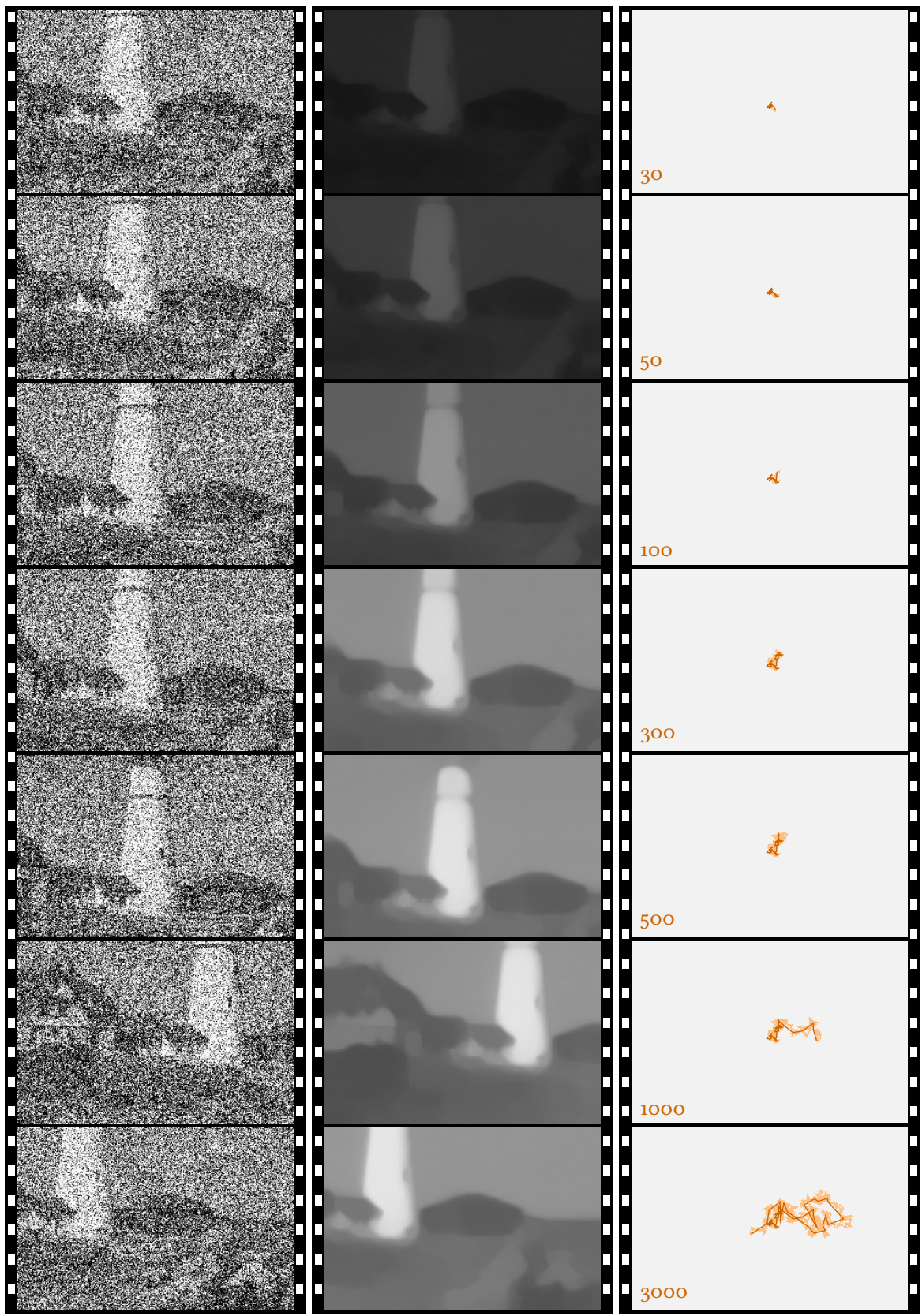}
    \caption{Lighthouse, POFB, approximately known displacement.}
    \label{fig:flow:lighthouse:pofb}
\end{figure}

\begin{figure}
    \centering%
    \tikzexternalenable%
    \def\yvalue{function_value}%
    \let\yvaluealt\undefined%
    \begin{subfigure}{0.49\textwidth}%
        \centering%
        \tikzsetnextfilename{valueplot_known_\yvalue}
\begin{tikzpicture}
    \begin{axis}[%
        width=\linewidth,
        xmode=log,
        xmin=1,xmax=1e4,
        scaled x ticks=false,
        x tick label style={/pgf/number format/fixed, /pgf/number format/set thousands separator={\,}},
        xminorticks=true,
        minor x tick num=1,
        yminorticks=true,
        log ticks with fixed point,
        minor y tick num=3,
        axis x line*=bottom,
        axis y line*=left,
        legend style={legend pos=north east,inner sep=0pt,outer sep=0pt,legend cell align=left,align=left,draw=none,fill=none,font=\scriptsize}
        ]

        \addplot [color=Set2-A, line width=1pt]
            table[x=iter,y=\yvalue]{\exponeknown.txt};
        \addlegendentry{Square, $\breve\rho=0$}

        \addplot [color=Set2-B, line width=1pt]
            table[x=iter,y=\yvalue]{\exptwoknown.txt};
        \addlegendentry{Lighthouse, $\breve\rho=0$}

        \addplot [color=Set2-C, line width=1pt]
            table[x=iter,y=\yvalue]{\expthreeknown.txt};
        \addlegendentry{Lighthouse,  $\breve\rho=100$}

        \addplot [color=Set2-D, line width=1pt]
            table[x=iter,y=\yvalue]{\exppofbknown.txt};
        \addlegendentry{Lighthouse, POFB}

        \ifdefined\yvaluealt
            \addplot [color=Set2-A, dashed, line width=1pt]
                table[x=iter,y=\yvaluealt]{\exponeknown.txt};

            \addplot [color=Set2-B, dashed, line width=1pt]
                table[x=iter,y=\yvaluealt]{\exptwoknown.txt};

            \addplot [color=Set2-C, dashed, line width=1pt]
                table[x=iter,y=\yvaluealt]{\expthreeknown.txt};

            \addplot [color=Set2-D, dashed, line width=1pt]
                table[x=iter,y=\yvaluealt]{\exppofbknown.txt};
        \else\fi
    \end{axis}

\end{tikzpicture}%
        \caption{Approximately known displacement}%
        \label{fig:flow:values:known}%
    \end{subfigure}%
    \begin{subfigure}{0.49\textwidth}%
        \centering%
        \pgfplotsset{every axis/.style=ignore legend}
        \tikzsetnextfilename{valueplot_unknown_\yvalue}
\begin{tikzpicture}
    \begin{axis}[%
        width=\linewidth,
        xmode=log,
        xmin=1,xmax=1e4,
        scaled x ticks=false,
        x tick label style={/pgf/number format/.cd,fixed,set thousands separator={\,}},
        xminorticks=true,
        minor x tick num=1,
        yminorticks=true,
        log ticks with fixed point,
        minor y tick num=3,
        axis x line*=bottom,
        axis y line*=left,
        legend style={legend pos=north west,inner sep=0pt,outer sep=0pt,legend cell align=left,align=left,draw=none,fill=none,font=\scriptsize}
        ]

        \addplot [color=Set2-A, line width=1pt]
            table[x=iter,y=\yvalue]{\exponeunknown.txt};
        \addlegendentry{Square, $\sigma=\sigma_{\max}$}

        \addplot [color=Set2-B, line width=1pt]
            table[x=iter,y=\yvalue]{\exptwounknown.txt};
        \addlegendentry{Lighthouse, $\sigma=1\sigma_{\max}$}

        \addplot [color=Set2-C, line width=1pt]
            table[x=iter,y=\yvalue]{\expthreeunknown.txt};
        \addlegendentry{Lighthouse, $\sigma=1000\sigma_{\max}$}

        \ifdefined\yvaluealt
            \addplot [color=Set2-A, dashed, line width=1pt]
                table[x=iter,y=\yvaluealt]{\exponeunknown.txt};

            \addplot [color=Set2-B, dashed, line width=1pt]
                table[x=iter,y=\yvaluealt]{\exptwounknown.txt};

            \addplot [color=Set2-C, dashed, line width=1pt]
                table[x=iter,y=\yvaluealt]{\expthreeunknown.txt};
        \else\fi

    \end{axis}

\end{tikzpicture}%
        \caption{Unknown displacement}%
        \label{fig:flow:values:unknown}%
    \end{subfigure}%
    \tikzexternaldisable%
    \caption{Iteration-wise objective values.}%
    \label{fig:flow:values}%
\end{figure}
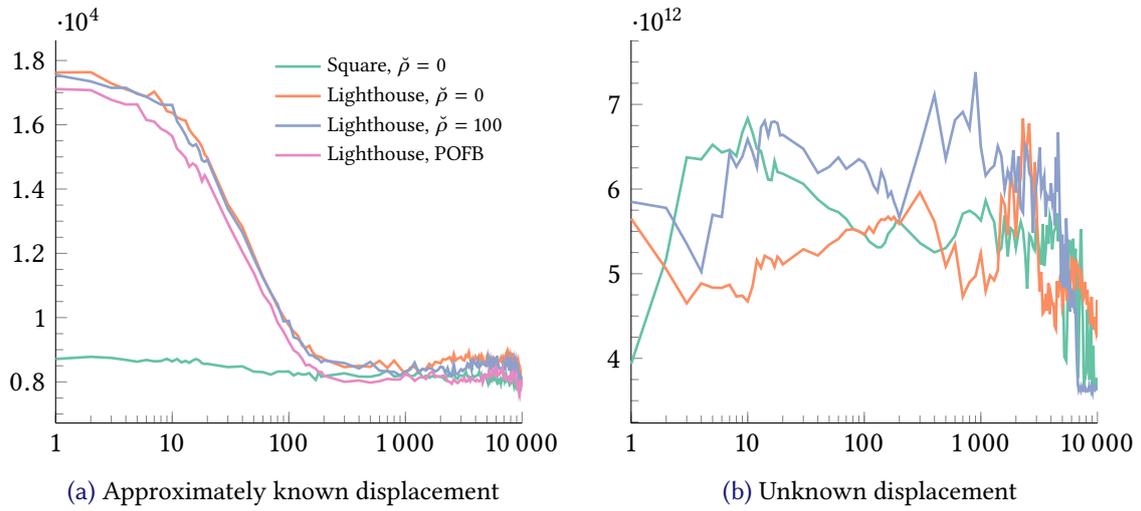

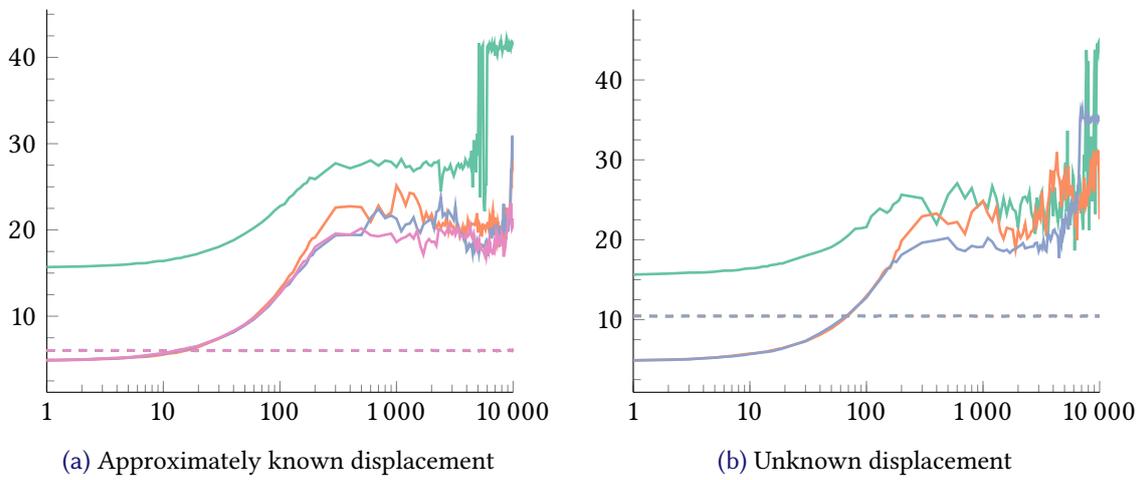
\begin{figure}
    \centering%
    \tikzexternalenable%
    \def\yvalue{psnr}%
    \def\yvaluealt{psnr_data}%
    \pgfplotsset{every axis/.style=ignore legend}
    \begin{subfigure}{0.49\textwidth}%
        \centering%
        \tikzsetnextfilename{valueplot_known_\yvalue}
\begin{tikzpicture}
    \begin{axis}[%
        width=\linewidth,
        xmode=log,
        xmin=1,xmax=1e4,
        scaled x ticks=false,
        x tick label style={/pgf/number format/fixed, /pgf/number format/set thousands separator={\,}},
        xminorticks=true,
        minor x tick num=1,
        yminorticks=true,
        log ticks with fixed point,
        minor y tick num=3,
        axis x line*=bottom,
        axis y line*=left,
        legend style={legend pos=north east,inner sep=0pt,outer sep=0pt,legend cell align=left,align=left,draw=none,fill=none,font=\scriptsize}
        ]

        \addplot [color=Set2-A, line width=1pt]
            table[x=iter,y=\yvalue]{\exponeknown.txt};
        \addlegendentry{Square, $\breve\rho=0$}

        \addplot [color=Set2-B, line width=1pt]
            table[x=iter,y=\yvalue]{\exptwoknown.txt};
        \addlegendentry{Lighthouse, $\breve\rho=0$}

        \addplot [color=Set2-C, line width=1pt]
            table[x=iter,y=\yvalue]{\expthreeknown.txt};
        \addlegendentry{Lighthouse,  $\breve\rho=100$}

        \addplot [color=Set2-D, line width=1pt]
            table[x=iter,y=\yvalue]{\exppofbknown.txt};
        \addlegendentry{Lighthouse, POFB}

        \ifdefined\yvaluealt
            \addplot [color=Set2-A, dashed, line width=1pt]
                table[x=iter,y=\yvaluealt]{\exponeknown.txt};

            \addplot [color=Set2-B, dashed, line width=1pt]
                table[x=iter,y=\yvaluealt]{\exptwoknown.txt};

            \addplot [color=Set2-C, dashed, line width=1pt]
                table[x=iter,y=\yvaluealt]{\expthreeknown.txt};

            \addplot [color=Set2-D, dashed, line width=1pt]
                table[x=iter,y=\yvaluealt]{\exppofbknown.txt};
        \else\fi
    \end{axis}

\end{tikzpicture}%
        \caption{Approximately known displacement}%
        \label{fig:flow:psnr:known}%
    \end{subfigure}%
    \begin{subfigure}{0.49\textwidth}%
        \centering%
        \tikzsetnextfilename{valueplot_unknown_\yvalue}
\begin{tikzpicture}
    \begin{axis}[%
        width=\linewidth,
        xmode=log,
        xmin=1,xmax=1e4,
        scaled x ticks=false,
        x tick label style={/pgf/number format/.cd,fixed,set thousands separator={\,}},
        xminorticks=true,
        minor x tick num=1,
        yminorticks=true,
        log ticks with fixed point,
        minor y tick num=3,
        axis x line*=bottom,
        axis y line*=left,
        legend style={legend pos=north west,inner sep=0pt,outer sep=0pt,legend cell align=left,align=left,draw=none,fill=none,font=\scriptsize}
        ]

        \addplot [color=Set2-A, line width=1pt]
            table[x=iter,y=\yvalue]{\exponeunknown.txt};
        \addlegendentry{Square, $\sigma=\sigma_{\max}$}

        \addplot [color=Set2-B, line width=1pt]
            table[x=iter,y=\yvalue]{\exptwounknown.txt};
        \addlegendentry{Lighthouse, $\sigma=1\sigma_{\max}$}

        \addplot [color=Set2-C, line width=1pt]
            table[x=iter,y=\yvalue]{\expthreeunknown.txt};
        \addlegendentry{Lighthouse, $\sigma=1000\sigma_{\max}$}

        \ifdefined\yvaluealt
            \addplot [color=Set2-A, dashed, line width=1pt]
                table[x=iter,y=\yvaluealt]{\exponeunknown.txt};

            \addplot [color=Set2-B, dashed, line width=1pt]
                table[x=iter,y=\yvaluealt]{\exptwounknown.txt};

            \addplot [color=Set2-C, dashed, line width=1pt]
                table[x=iter,y=\yvaluealt]{\expthreeunknown.txt};
        \else\fi

    \end{axis}

\end{tikzpicture}%
        \caption{Unknown displacement}%
        \label{fig:flow:psnr:unknown}%
    \end{subfigure}%
    \tikzexternaldisable%
    \caption{Iteration-wise PSNR. The dashed lines indicate the PSNR for the noisy data corresponding to the experiment of the solid line of the same colour. Legend in  \cref{fig:flow:values:known}.}%
    \label{fig:flow:psnr}%
\end{figure}

\begin{figure}
    \centering%
    \tikzexternalenable%
    \def\yvalue{ssim}%
    \def\yvaluealt{ssim_data}%
    \pgfplotsset{every axis/.style=ignore legend}
    \begin{subfigure}{0.49\textwidth}%
        \centering%
        \tikzsetnextfilename{valueplot_known_\yvalue}
\begin{tikzpicture}
    \begin{axis}[%
        width=\linewidth,
        xmode=log,
        xmin=1,xmax=1e4,
        scaled x ticks=false,
        x tick label style={/pgf/number format/fixed, /pgf/number format/set thousands separator={\,}},
        xminorticks=true,
        minor x tick num=1,
        yminorticks=true,
        log ticks with fixed point,
        minor y tick num=3,
        axis x line*=bottom,
        axis y line*=left,
        legend style={legend pos=north east,inner sep=0pt,outer sep=0pt,legend cell align=left,align=left,draw=none,fill=none,font=\scriptsize}
        ]

        \addplot [color=Set2-A, line width=1pt]
            table[x=iter,y=\yvalue]{\exponeknown.txt};
        \addlegendentry{Square, $\breve\rho=0$}

        \addplot [color=Set2-B, line width=1pt]
            table[x=iter,y=\yvalue]{\exptwoknown.txt};
        \addlegendentry{Lighthouse, $\breve\rho=0$}

        \addplot [color=Set2-C, line width=1pt]
            table[x=iter,y=\yvalue]{\expthreeknown.txt};
        \addlegendentry{Lighthouse,  $\breve\rho=100$}

        \addplot [color=Set2-D, line width=1pt]
            table[x=iter,y=\yvalue]{\exppofbknown.txt};
        \addlegendentry{Lighthouse, POFB}

        \ifdefined\yvaluealt
            \addplot [color=Set2-A, dashed, line width=1pt]
                table[x=iter,y=\yvaluealt]{\exponeknown.txt};

            \addplot [color=Set2-B, dashed, line width=1pt]
                table[x=iter,y=\yvaluealt]{\exptwoknown.txt};

            \addplot [color=Set2-C, dashed, line width=1pt]
                table[x=iter,y=\yvaluealt]{\expthreeknown.txt};

            \addplot [color=Set2-D, dashed, line width=1pt]
                table[x=iter,y=\yvaluealt]{\exppofbknown.txt};
        \else\fi
    \end{axis}

\end{tikzpicture}%
        \caption{Approximately known displacement}%
        \label{fig:flow:ssim:known}%
    \end{subfigure}%
    \begin{subfigure}{0.49\textwidth}%
        \centering%
        \tikzsetnextfilename{valueplot_unknown_\yvalue}
\begin{tikzpicture}
    \begin{axis}[%
        width=\linewidth,
        xmode=log,
        xmin=1,xmax=1e4,
        scaled x ticks=false,
        x tick label style={/pgf/number format/.cd,fixed,set thousands separator={\,}},
        xminorticks=true,
        minor x tick num=1,
        yminorticks=true,
        log ticks with fixed point,
        minor y tick num=3,
        axis x line*=bottom,
        axis y line*=left,
        legend style={legend pos=north west,inner sep=0pt,outer sep=0pt,legend cell align=left,align=left,draw=none,fill=none,font=\scriptsize}
        ]

        \addplot [color=Set2-A, line width=1pt]
            table[x=iter,y=\yvalue]{\exponeunknown.txt};
        \addlegendentry{Square, $\sigma=\sigma_{\max}$}

        \addplot [color=Set2-B, line width=1pt]
            table[x=iter,y=\yvalue]{\exptwounknown.txt};
        \addlegendentry{Lighthouse, $\sigma=1\sigma_{\max}$}

        \addplot [color=Set2-C, line width=1pt]
            table[x=iter,y=\yvalue]{\expthreeunknown.txt};
        \addlegendentry{Lighthouse, $\sigma=1000\sigma_{\max}$}

        \ifdefined\yvaluealt
            \addplot [color=Set2-A, dashed, line width=1pt]
                table[x=iter,y=\yvaluealt]{\exponeunknown.txt};

            \addplot [color=Set2-B, dashed, line width=1pt]
                table[x=iter,y=\yvaluealt]{\exptwounknown.txt};

            \addplot [color=Set2-C, dashed, line width=1pt]
                table[x=iter,y=\yvaluealt]{\expthreeunknown.txt};
        \else\fi

    \end{axis}

\end{tikzpicture}%
        \caption{Unknown displacement}%
        \label{fig:flow:ssim:unknown}%
    \end{subfigure}%
    \tikzexternaldisable%
    \caption{Iteration-wise SSIM. The dashed lines indicate the SSIM for the noisy data corresponding to the experiment of the solid line of the same colour. Legend in  \cref{fig:flow:values:known}.}%
    \label{fig:flow:ssim}%
\end{figure}
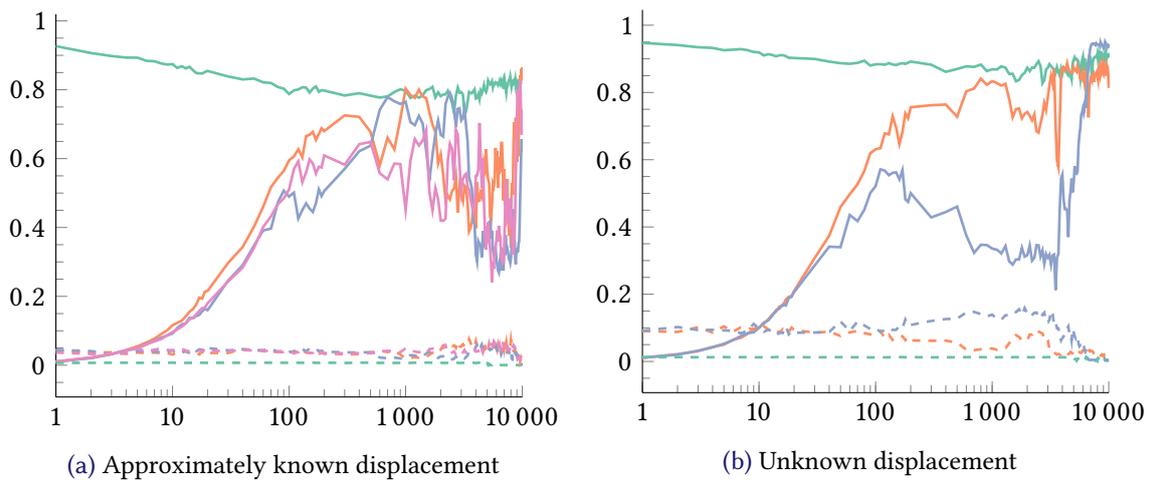

\begin{figure}%
    \dostrip[true]{\exponeunknown}
    \input{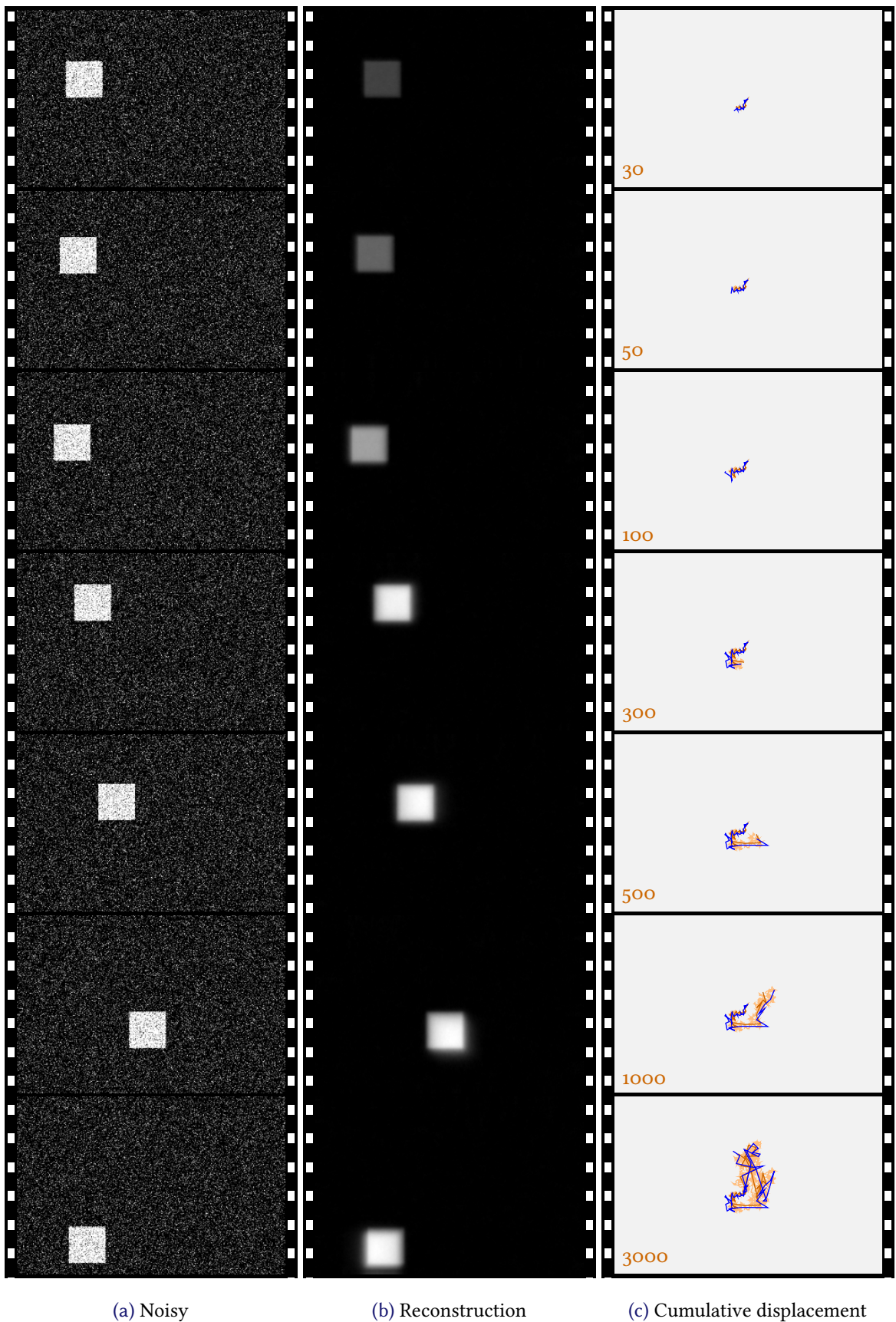}
    \caption{Square, POPD, unknown displacement, $\rho=\EXPPARAMphantomrho$.
    The blue line in (c) indicates the estimated displacement field.}
    \label{fig:flow:square:unknown}
\end{figure}

\begin{figure}%
    \dostrip[true]{\exptwounknown}
    \input{img/\exptwounknown_params.tex}
    \caption{Lighthouse, POPD, unknown displacement, $\rho=\EXPPARAMphantomrho$.
    The blue line in (c) indicates the estimated displacement field.}
    \label{fig:flow:lighthouse:unknown}
\end{figure}

\begin{figure}%
    \dostrip[true]{\expthreeunknown}
    \input{img/\expthreeunknown_params.tex}
    \caption{Lighthouse, POPD, unknown displacement, $\rho=\EXPPARAMphantomrho$.
    The blue line in (c) indicates the estimated displacement field.}
    \label{fig:flow:lighthouse:unknown-largesigma}
\end{figure}

\end{document}